\documentclass[12pt]{amsart}

\usepackage{fullpage}

\usepackage[colorlinks, linkcolor=blue, citecolor=blue, urlcolor=red!80!gray, pagebackref]{hyperref}
\usepackage{amssymb, amsmath, xcolor} 
\usepackage{xspace}
\usepackage{import}
\usepackage{tikz}
\usetikzlibrary{backgrounds}
\usepackage{ytableau}
\usepackage{xcolor}
\usepackage{array}
\tikzset{filled/.style={minimum width=5pt,inner sep=0pt,circle,fill=black}}
\usepackage{subcaption}

\usepackage[indent]{parskip}
\usepackage{mathtools}
\usepackage[utf8]{inputenc}
\usepackage{float}
\usepackage{longtable}
\usepackage{todonotes}
\usepackage[square,numbers]{natbib}
\usepackage{comment}
\newtheorem{theorem}{Theorem}[section]
\newtheorem{lemma}[theorem]{Lemma}
\newtheorem{corollary}[theorem]{Corollary}

\theoremstyle{definition}
\newtheorem{definition}[theorem]{Definition}
\newtheorem{example}[theorem]{Example}

\theoremstyle{remark}
\newtheorem{remark}[theorem]{Remark}

\numberwithin{equation}{section}
\numberwithin{figure}{section}

\renewcommand{\mu}{\Lambda}

\renewcommand{\mod}{\operatorname{mod}}
\newcommand{\N}{\mathbb{N}}
\newcommand{\Z}{\mathbb{Z}}

\newcommand{\SYT}{\operatorname{SYT}}
\newcommand{\lsmash}{
    \raisebox{0.6ex}{
        \begin{tikzpicture}[baseline=0ex]
            \draw (0,0) circle [radius=0.7ex];
            \draw[->, line width=0.3pt] (0.45ex,0) -- (-0.45ex,0);
        \end{tikzpicture}
    }
} 
\newcommand{\biglsmash}{
  \mathop{
    \raisebox{0.6ex}{
        \begin{tikzpicture}[baseline=0ex]
            \draw[line width=0.8pt] (0,0) circle [radius=1.6ex];
            \draw[->, line width=0.8pt] (0.9ex,0) -- (-0.9ex,0);
        \end{tikzpicture}
    }
  }\limits
} 
\newcommand{\jdt}{\operatorname{jdt}}
\newcommand{\prom}{\jdt}

\newcommand{\withlines}[1]{%
  \begin{tikzpicture}[baseline=(T.base)]
    \node[inner sep=-0.18pt, outer sep=0] (T) {#1};
    \draw[line width=.4pt] (T.north west) -- ([xshift=1em]T.north east);
    \draw[line width=.4pt] (T.south west) -- ([xshift=1em]T.south east);
  \end{tikzpicture}%
}

\newcommand{\withlinesvert}[1]{%
  \begin{tikzpicture}[baseline=(T.base)]
    \node[inner sep=-0.18pt, outer sep=0] (T) {#1};
    \draw[line width=.4pt] (T.north west) -- ([yshift=-1em]T.south west);
    \draw[line width=.4pt] (T.north east) -- ([yshift=-1em]T.south east);
  \end{tikzpicture}%
}

\newcommand{\withlinesboth}[1]{%
  \begin{tikzpicture}[baseline=(T.base)]
    \node[inner sep=-0.18pt, outer sep=0] (T) {#1};
    \draw[line width=.4pt] (T.north west) -- ([xshift=1em]T.north east);
    \draw[line width=.4pt] (T.north west) -- ([yshift=-1em]T.south west);
  \end{tikzpicture}%
}

\title[Jeu de taquin on infinite Young tableaux]{Chaotic and periodic behavior of jeu de taquin on infinite Young tableaux}

\author[Herden, Hunziker, Meddaugh, Sepanski]{
    Daniel Herden,
    Markus Hunziker,
    Jonathan Meddaugh,
    Mark R. Sepanski,
    Lauren Engelthaler,
    Jonathan Feigert,
    Teresa Gulding,
    Jes\'us Hern\'andez-Rodr\'iguez,
    Mitchell Minyard,
    Kyle Rosengartner,
    Dona~Ishara~N.~Saparamadu,
    Robert Stansbury
    }
    
\date{\today}

\address{
	Department of Mathematics,
	Baylor University,
	Sid Richardson Building,
	1410 S.~4th Street,
	Waco, TX 76706, USA}

\email{
    Daniel\_Herden@baylor.edu,
    Markus\_Hunziker@baylor.edu,
    Jonathan\_Meddaugh@baylor.edu,
    Mark\_Sepanski@baylor.edu,
    Lauren\_Engelthaler1@baylor.edu,
    Jonathan\_Feigert1@baylor.edu,
    Teresa\_Gulding1@baylor.edu,
    Jesus\_Hernandez6@baylor.edu,
    Mitch\_Minyard1@baylor.edu,
    Kyle\_Rosengartner1@baylor.edu,
    Ishara\_Saparamadu1@baylor.edu,
    Robbie\_Stansbury1@baylor.edu
    }

\thanks{The first author was supported by Simons Foundation grant MPS-TSM-00007788.
	The third author was supported by a grant from the Simons Foundation (960812, JM)}

\begin{document}

\keywords{jeu de taquin, infinite Young tableaux, promotion}
\subjclass[2020]{Primary: 05E10; Secondary: 20B30}

\begin{abstract} 
    Young tableaux are fundamental objects in algebraic combinatorics and
    representation theory, with operations such as promotion and jeu de taquin
    playing a central role in their structure and applications.
    While these operations are well understood for finite tableaux, their behavior
    on infinite tableaux has so far been studied mainly within probabilistic
    frameworks.
    In this paper, we investigate jeu de taquin on infinite standard Young tableaux
    from a purely combinatorial and dynamical point of view.
    We analyze the action of jeu de taquin on infinite shapes, describe the
    structure of inverse images, and classify tableaux exhibiting periodic,
    pre-periodic, and recurrent behavior.
    We also introduce a natural metric on the space of infinite tableaux and show
    that jeu de taquin defines a chaotic dynamical system in the sense of Devaney.
    These results extend classical tableau theory to infinite
    settings and identify connections between combinatorial
    dynamics and infinite representation-theoretic structures.
\end{abstract}

\maketitle
\setcounter{tocdepth}{1}
 \tableofcontents

\section{Introduction}
\emph{Young diagrams} and \emph{Young tableaux} are classical combinatorial objects
that play a central role in algebraic combinatorics, representation theory,
and algebraic geometry \cite{fulton1997young}.
Among the most important operations on Young tableaux are \emph{promotion}
and the closely related process of \emph{jeu de taquin}, both introduced by
Sch\"utzenberger \cite{Schutzenberger-Promotion-Original,Schutzenberger1977}.
Jeu de taquin provides a fundamental mechanism for rectifying skew tableaux
and lies at the heart of many structural results in tableau theory.
A systematic modern treatment of promotion and its interaction with
evacuation was given by Stanley \cite{MR2515772}.

Since their introduction, promotion and jeu de taquin have appeared in a wide
range of mathematical contexts.
They play a key role in Schubert calculus and $K$-theoretic Schubert calculus
\cite{MR3200334,levinson2017one,purbhoo2013wronskians,MR2491941,MR2806593},
in the theory of classical Lie groups
\cite{arnal2018sliding,MR1466956},
in crystal theory and related representation-theoretic frameworks
\cite{pfannerer2020promotion,van1998analogue},
and in the study of cyclic sieving phenomena and representation theory
\cite{alexandersson2021skew,petersen2009promotion,rhoades2010cyclic}.
They also arise naturally in integrable systems and interacting particle models
\cite{iwao2019jeu}.
More recently, promotion has been studied from an explicitly dynamical
viewpoint, notably through the introduction of promotion permutations and
promotion matrices, which encode the global orbit structure of promotion and
have led to new constructions in web combinatorics
\cite{gaetz2023promotion,gaetz2025web}.
These developments fit into the broader program of
\emph{dynamical algebraic combinatorics}, which studies natural combinatorial
operators via their dynamical properties \cite{Roby2014}.

While Young tableaux are most often considered for finite shapes, the notion
of a standard Young tableau extends naturally to infinite diagrams.
Infinite tableaux were first studied systematically by Vershik and Kerov in
their work on the Plancherel measure and asymptotic representation theory of
the infinite symmetric group
\cite{kerov1986characters,vervshik1977asymptotic,vershik1981asymptotic}.
Within this probabilistic framework, infinite tableaux and related algorithms
such as Robinson--Schensted--Knuth (RSK) exhibit rich limiting behavior.

Jeu de taquin on infinite tableaux has so far been investigated primarily from
this probabilistic perspective.
In particular, Romik and \'{S}niady developed a detailed theory of jeu de taquin
dynamics on random infinite Young tableaux, relating sliding paths to
interacting particle systems and second class particles
\cite{RomikSniady2015}.
Related work has further explored limit shapes and asymptotic behavior of
sliding and evacuation paths for random tableaux
\cite{MaslankaSniady2019},
as well as structural properties of RSK and jeu de taquin under Kerov--Vershik
measures \cite{Sniady2014}.
Outside this probabilistic setting, however, the behavior of jeu de taquin as
an \emph{iterated combinatorial dynamical system} on infinite tableaux has
received comparatively little attention.
In contrast to this probabilistic perspective, our approach treats jeu de
taquin as a deterministic combinatorial dynamical system and focuses on
qualitative dynamical behavior---such as periodicity, recurrence, and
chaos---rather than on asymptotic laws arising from specific probability
measures.

This paper studies jeu de taquin on infinite standard Young tableaux from a
combinatorial and dynamical point of view, independent of probabilistic
assumptions.
We analyze the action of jeu de taquin as a self-map on spaces of infinite
tableaux, focusing on its global dynamical behavior.
In particular, we classify and construct tableaux exhibiting periodic and
pre-periodic behavior, describe the structure of inverse images under jeu de
taquin, and study recurrent phenomena in settings where genuine periodicity is
impossible.
We also introduce a natural metric on the space of infinite tableaux and show
that jeu de taquin exhibits chaotic behavior in the sense of Devaney.
Unlike most previously studied combinatorial operators, which act on finite
sets or exhibit rigid orbit structures, jeu de taquin on infinite tableaux
gives rise to a noncompact dynamical system in which chaos arises from the
interaction of local combinatorial rules with infinite global structures.

The paper is organized as follows.
In Section~\ref{sec: notation}, we introduce the notation and terminology used
throughout, including infinite shapes and the left smash construction.
Section~\ref{sec: basic} establishes foundational properties of jeu de taquin
on infinite tableaux, including a detailed analysis of inverse images.
In Section~\ref{sec: periodic}, we characterize periodic and pre-periodic
tableaux and describe their structure.
Section~\ref{sec: chaotic} introduces a natural metric on spaces of infinite
tableaux and proves that jeu de taquin is chaotic in the sense of Devaney.
In Section~\ref{sec: recurrence}, we construct and analyze examples of
recurrent infinite tableaux.
Finally, Section~\ref{sec: concluding remarks} discusses open problems and
directions for future research.

\section{Notation} \label{sec: notation}

\subsection{Tableaux}

See \cite{fulton1997young} and \cite{Sagan2001} for references.
We write $\N$ for the set of positive integers, and for $a, b \in \N$ with $a \leq b$, we write 
\[[a, b]=\{c\in\N \mid a\leq c\leq b\}.\]
For $n \in \N$, a \emph{partition} of $n$ is a weakly decreasing finite sequence $\lambda = (\lambda_1, \ldots, \lambda_\ell)$ of positive integers, such that $\sum_{i=1}^\ell \lambda_i = n$.
We call $n$ the \emph{size} of $\lambda$, which we express by writing 
\[\lambda \vdash n \quad \mbox{ or } \quad |\lambda|=n. \]
We will often use exponents to denote repeated integers in a partition; for example, $(4^3,2,1^2)$ is shorthand for $(4,4,4,2,1,1)$.
The \emph{Young diagram of shape $\lambda$} is the left-justified array with $\lambda_i$ boxes in its $i$th row, counting from the top. 
As an example, the Young diagram of shape $\lambda = (3,2,1,1)$ is given in Figure~\ref{subfig:YD}.
For brevity, we often refer to the ``rows/columns of $\lambda$'' or similar phrases, where it is understood that we are speaking of the Young diagram of shape $\lambda$. We index the boxes of a Young diagram by ordered pairs $(i,j)$ referring to the $i$th row and $j$th column.

Recall the \emph{conjugate partition} of $\lambda$, denoted by $\lambda' = (\lambda'_1, \ldots, \lambda'_m)$, where  $m = \lambda_1$, and where $\lambda'_j$ denotes the length of the $j$th column of the Young diagram of shape $\lambda$, or equivalently
    \[
    \lambda'_j= |\{i \mid \lambda_i \geq j\}|.
    \]
For example, if $\lambda = (3,2,1,1)$ is the partition with Young diagram shown in Figure~\ref{subfig:YD}, then we have $\lambda' = (4,2,1)$.

For $\lambda \vdash n$, a \emph{standard Young tableau of shape $\lambda$} is a Young diagram of shape $\lambda$, where the boxes are filled with distinct entries from the alphabet $[1,n]$ such that every row and column is increasing.
See Figure~\ref{subfig:SYT} for an example.
Let 
\[ \SYT(\lambda)\]
denote the set of all standard Young tableaux of shape $\lambda$. If $S\in\SYT(\lambda)$, we write $S_{i,j}$ for the entry in the $(i,j)$-th box of $S$. Occasionally, we will also write $|S|=n$ when $\lambda \vdash n$.

More generally, if $A\subseteq \Z$, we write 
\[ \SYT(\lambda, A) \] 
for the set of Young diagrams of shape $\lambda$ where the boxes are filled with distinct entries from the alphabet $A$ such that every row and column is increasing, and we refer to the elements of $\SYT(\lambda, A)$ as \emph{standard Young tableaux of shape $\lambda$ with entries from $A$}.

\begin{figure}[H]
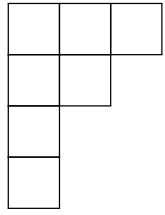
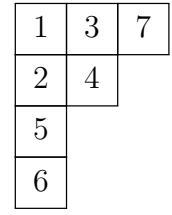

     \centering
     \ytableausetup{boxsize=1.6em} 
     \begin{subfigure}[b]{0.4\textwidth}
         \centering
         \ydiagram{3,2,1,1}
         \caption{The Young diagram of shape $\lambda$.}
         \label{subfig:YD}
     \end{subfigure}
     \qquad
     \begin{subfigure}[b]{0.4\textwidth}
         \centering
         \ytableaushort{137,24,5,6}
         \caption{An element of $\SYT(\lambda)$.}
         \label{subfig:SYT}
     \end{subfigure}
     \caption{Example where $\lambda = (3,2,1,1)$.
     Note that $\lambda' = (4,2,1)$.}
     \label{fig:YD and SYT}
\end{figure}


For $n\in \N,m\in\Z$, $\lambda \vdash n$, and $S\in\SYT(\lambda, \Z)$, we define the \emph{$m$-shift} of $S$, written
\[ S+m, \]
to be the standard Young tableau obtained from $S$ by adding $m$ to each entry of $S$. 

\subsection{Skew Tableaux}

If $n,n'\in \N$, $\lambda \vdash n$, and $\nu \vdash n'$ with $\nu_i \leq \lambda_i$ for all $1\leq i \leq \nu_1'$, then $\lambda$ is said to \emph{contain} $\nu$.
Then the \emph{skew Young diagram of skew shape} $\sigma = \lambda / \nu$ consists of the set of boxes belonging to the Young diagram of $\lambda$ but not to that of $\nu$. 
In this context $\nu$ may be referred to as the \emph{inner partition}.
Write $\sigma_i = \lambda_i - \nu_i$ for the length of the $i$th row of $\sigma$, where we set $\nu_i=0$ for all $i > \nu_1'$, and let $|\sigma|=\sum_i \sigma_i$. We say that $\sigma$ is \emph{weakly decreasing} if $\sigma_i \geq \sigma_{i+1}$ for all $i$. We write $\sigma' = \lambda' / \nu'$ for the \emph{conjugate skew Young diagram} and write
\[ \ell(\sigma) = \max\{i \,|\, \lambda_i > \nu_i \} \]
for the last index where the two partitions $\lambda$ and $\nu$ differ, i.e., the overall height of $\sigma$ (possibly counting intermediate empty rows).

Further, we write
$\SYT(\sigma, A)$ for the set of skew Young diagrams of shape $\sigma$ where the boxes are filled with distinct entries from the alphabet $A$ such that every row and column is increasing. We refer to the elements of $\SYT(\sigma, A)$ as \emph{standard skew Young tableaux of shape $\sigma$ with entries from $A$}. If $A=[1,|\sigma|]$, we may choose to simply write $\SYT({\sigma})$. 
As with tableaux, 
for $m\in\Z$ and $S\in\SYT(\sigma, \Z)$, we define the \emph{$m$-shift} of $S$, written $S+m$,
to be the standard skew Young tableau obtained from $S$ by adding $m$ to each entry of $S$. 

Finally, if $\sigma(j)$ are skew diagrams and $S(j) \in \SYT(\sigma(j), \N)$ for $j=1,2$, we say that $S(1)$ and $S(2)$ are \emph{row equivalent}, written $S(1) \sim S(2)$, if, for all $i$, 
the set of entries in the $i$th rows of $\sigma(1)$ and of $\sigma(2)$ coincide. Note that this requires
$\sigma(1)_i = \sigma(2)_i$ for all $i$.



\begin{figure}[H]
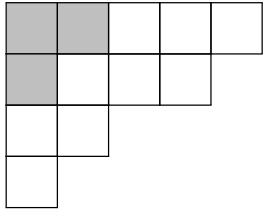
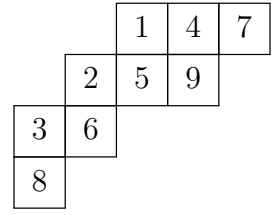

    \centering
    \ytableausetup{boxsize=1.6em} 
    \begin{subfigure}[b]{0.42\textwidth}
        \centering
        \begin{ytableau}
            *(lightgray) & *(lightgray) &  &  &  \\
            *(lightgray) &  &  &  \\
             &  \\
             \,
        \end{ytableau}
        \caption{Skew Young diagram $\sigma=\lambda/\nu$ with $\lambda=(5,4,2,1)$ and $\nu=(2,1)$ (shaded).}
        \label{subfig:skew-diagram-shaded}
    \end{subfigure}
    \hfill
    \begin{subfigure}[b]{0.42\textwidth}
        \centering
        \begin{ytableau}
            \none & \none & 1 & 4 & 7 \\
            \none & 2 & 5 & 9 \\
            3 & 6 \\
            8
        \end{ytableau}
        \caption{An element of $\SYT(\sigma)$.}
        \label{subfig:skew-SYT}
    \end{subfigure}
    \caption{A skew Young diagram with shaded inner partition and a corresponding standard skew Young tableau.}
    \label{fig:skew-and-SYT-shaded}
\end{figure}

         

Given $S\in\SYT(\lambda, \Z)$  or $S\in \SYT(\sigma,\Z)$, the \emph{Knuth word} or \emph{reading word} of
$S$ is the word given by concatenating the rows of $S$ from the bottom row to the top row, i.e.,
by reading off the entries of each row of $S$, left-to-right, from bottom-to-top. 
As an example, the Knuth word assigned to the standard skew Young tableau in Figure~\ref{subfig:skew-SYT} is $836259147$.
An \emph{elementary Knuth transformation} consists of applying either the equivalence $yzx \equiv yxz$ for $x<y\leq z$ or $xzy \equiv zxy$ for $x \leq y < z$ to three consecutive letters of a Knuth word, though, in our setting we will always have strict inequalities. 
Two words are called \emph{Knuth equivalent} if one can be obtained from the other by a sequence of elementary Knuth transformations.

\subsection{Infinite Tableaux}

In this paper, we consider three types of infinite generalizations of shapes, Young diagrams, and Young tableaux. For $l\in\N$, we define the following.
\begin{enumerate}
    \item The \emph{infinite Young diagram of shape} $(\infty^l)=(\overbrace{\infty, \ldots, \infty}^{l})$ is the left-justified array with boxes indexed by $[1, l] \times \N$, i.e., $l$ countably infinite rows stretching to the right.
    \item The \emph{infinite Young diagram of shape} $(l^\infty) = (l,l,l,\ldots)$ is the left-justified array with boxes indexed by $\N \times [1, l]$, i.e., $l$ countably infinite columns stretching downwards.
    \item  The \emph{infinite Young diagram of shape} $(\infty^\infty)=(\infty,\infty,\infty,\ldots)$ is the left-justified array with boxes indexed by $\N \times \N$, i.e., countably infinite rows and columns stretching to the right and downwards.
\end{enumerate}
We call the shapes defined above \emph{infinite shapes}, and we define an \emph{infinite standard Young tableau} of shape $\mu$ to be a bijective $\N$-filling (i.e., every positive integer appears exactly once as entry) of all boxes of the infinite Young diagram $\mu$ so that every row and column is increasing. We write 
\[\SYT(\mu)\] 
for the set of all infinite standard Young tableaux of shape $\mu$. 



\begin{figure}[H]
    \centering
    \ytableausetup{boxsize = 1.6em}

    \begin{subfigure}[b]{0.4\textwidth}
        \centering
        \withlines{
        \begin{ytableau}
            1 & 2 & 6 & 8 & \none[\cdots]\\
            3 & 4 & 7 & 11 & \none[\cdots]\\
            5 & 9 & 10 & 12 & \none[\cdots]
        \end{ytableau}}
        \caption{An element of $\SYT((\infty^3))$.}
        \label{subfig:Horizontal}
    \end{subfigure}
    \begin{subfigure}[b]{0.4\textwidth}
        \centering
        \withlinesvert{
        \begin{ytableau}
            1 & 3 & 4\\
            2 & 5 & 6\\
            \none[\vdots] & \none[\vdots] & \none[\vdots]
        \end{ytableau}}
        \caption{An element of $\SYT((3^\infty))$.}
        \label{subfig:Vertical}
    \end{subfigure}

    \vspace{1em} 

    \begin{subfigure}[c]{0.4\textwidth}
        \centering
        \withlinesboth{
        \begin{ytableau}
            1 & 3 & 4 & \none[\cdots]\\
            2 & 6 & 7 & \none[\cdots]\\
            5 & 9 & 13 & \none[\cdots]\\
            \none[\vdots] & \none[\vdots] & \none[\vdots] & \none[\ddots]
        \end{ytableau}}
        \caption{An element of $\SYT((\infty^\infty))$.}
        \label{subfig:Bi-infinite}
    \end{subfigure}

    \caption{Examples of infinite standard Young tableaux.}
    \label{fig: Infinite SYT}
\end{figure}

As with finite Young diagrams, one can analogously define \emph{conjugate partitions} of infinite Young diagrams, and we use similar notation. For example, if $\mu =(\infty^l)$, then $\mu' = (l^\infty)$.
Similarly, if $T\in \SYT{((\infty^l))}$, then $T' \in \SYT{((l^\infty))}$ denotes the \emph{conjugate tableau} (transpose) of $T$ which results from flipping the boxes and entries of $T$ along its main diagonal.
The map $\Phi: \SYT{((\infty^l))} \to \SYT{((l^\infty))}$ given by $\Phi(T)=T'$ is a bijection, and it will be straightforward to show that $\Phi \circ \jdt = \jdt \circ \Phi$ with respect to jeu de taquin.
Hence, as statements about the infinite shape $(l^\infty)$ can be obtained from corresponding statements about $(\infty^l)$, we will focus on the latter in this paper. 

Suppose $n,l\in\N$, $\lambda \vdash n$, $S\in\SYT(\lambda)$, $\mu$ is a finite or infinite shape, and $T\in\SYT(\mu)$. 
We say that $\mu$ \emph{contains} $\lambda$ if $\lambda_i \leq \mu_i$ for $1\leq i \leq \lambda_1'$. For $\mu$ of the form $(\infty^\infty)$, there is no constraint. For $\mu$ of the form $(\infty^l)$, the only constraint is $\lambda_1' \leq \mu'_1 = l$.
In the case that $\mu$ contains $\lambda$ and $T_{i,j} = S_{i,j}$ for all boxes $(i,j)$ of $\lambda$, we say that $T$ \emph{contains} $S$, written $T|_\lambda = S$, and we may refer to this by either saying that $T$ \emph{restricts} to $S$ or that $T$ is an \emph{extension} of $S$. 
We use similar notation for a skew shape $\sigma = \lambda / \nu$ with $\nu$ contained in $\lambda$, $\lambda$ contained in $\mu$, and $S \in \SYT(\sigma, A)$. Finally, for $\mu$ an infinite shape,  $T\in\SYT(\mu)$, and $a,b\in\N$ with $a\le b$, the standard skew Young tableau $T[a,b]$ results from restricting the tableau $T$ to the boxes with entries from $[a,b]$.



\subsection{The Left Smash}

The following construction will be used repeatedly throughout this paper when studying infinite standard Young tableaux of shape $(\infty^l)$. The terminology arises from imagining a sequence of standard Young tableaux placed successively to the right, each one having its entries incremented by the total size of the previous tableaux and then being left-justified by sliding (smashing) everything to the left.


\begin{definition}[\emph{Left Smash} for $(\infty^l)$] \label{def: smash infty^l}
    Fix $l\in\N$ and the infinite shape $\mu = (\infty^l)$.
    For each $k\in\N$, let
    \begin{itemize}
        \item $n_k\in\N$,
        \item $\lambda(k) \vdash n_k$,
        \item $S(k) \in \SYT(\lambda(k))$
    \end{itemize}
    with an infinite number of $\lambda(k)$ having maximal height, i.e., $|\{ k \in \N \mid \lambda(k)_1' = l \}| = \infty$.
   
    Iteratively define $T\in\SYT(\mu)$, written
        \[ T = \biglsmash_{k=1}^\infty S(k),\]
    as follows. Begin by inserting $S(1)$ into $T$ by defining 
    \[ T_{i,j}=S(1)_{i,j} \]
    for all boxes $(i,j)$ in $\lambda(1)$. 

    Next, insert $S(2)$ into $T$ by incrementing the entries of $S(2)$ by $n_1$ and then placing them to the right of where $S(1)$ was already inserted. That is, define
     \[ T_{i, \, 
          \lambda(1)_i + j}   =   S(2)_{i,j} + n_1\]
    for all boxes $(i,j)$ in $\lambda(2)$.
    
    In general, assuming $S(1), \ldots, S(k-1)$ have already been inserted into $T$, insert $S(k)$ by defining
    \[ T_{i, \, 
          \lambda(1)_i+\ldots+\lambda(k-1)_i + j}   =   S(k)_{i,j} + n_1 + \ldots + n_{k-1}\]
    for all boxes $(i,j)$ in $\lambda(k)$. 
    
    If the above procedure is stopped at $k=k_0$, we write
    \[ \biglsmash_{k=1}^{k_0} S(k) \qquad \text{or} \qquad S(1) \lsmash \ldots \lsmash S(k_0)\]
    for the resulting standard Young tableau.
\end{definition}

Note that in Definition \ref{def: smash infty^l}, the condition $|\{ k \in \N \mid \lambda(k)_1' = l \}| = \infty$ ensures that the left smash results in a complete filling of all of $\mu$, and filling in a sequence of standard Young tableaux $S(k)$ ensures that the incrementing results in a standard Young tableau $T$ with increasing columns.

As an illustration, consider the example in Figure \ref{figure: left smash example} where $l = 3$.
\begin{figure}[H] 
    \centering
    \ytableausetup{boxsize = 1.6em}
    \[
    S(1) = \ytableaushort[*(gray)]{12,34,5}\,\,\,\,\,\,\,\,\,\, 
    S(2) = \ytableaushort[*(lightgray)]{13,2,4}\,\,\,\,\,\,\,\,\,\, 
    S(k) = \ytableaushort{13,25,4} \text{ for } k\geq 3.
    \]

    \vspace{0.8em} 

    \[
    \biglsmash_{k=1}^\infty S(k) = 
    \withlines{%
        \begin{ytableau}
            *(gray) 1 & *(gray) 2 & *(lightgray) 6 & *(lightgray) 8 & 10 & 12 & \none[\cdots]\\
            *(gray) 3 & *(gray) 4 & *(lightgray) 7 & 11 & 14 & 16 & \none[\cdots]\\
            *(gray) 5 & *(lightgray) 9 & 13 & 18 & 23 & 28 & \none[\cdots]
        \end{ytableau}%
    }
    \]
    \caption{Example of left smash for $(\infty^3)$.}
    \label{figure: left smash example}
\end{figure}



\begin{remark}[Extension of left smash to skew shapes] \label{rem: can smash with skew too}
    In Definition \ref{def: smash infty^l}, we may generalize the construction to allow smashing with standard skew Young tableaux under slightly adjusted requirements. Namely, for each $k\in \N$, let
    \begin{itemize}
    \item $n_k\in\N$,
        \item $\lambda(k)$ is a skew Young diagram with $|\lambda(k)|=n_k$ such that $|\{ k \in \N \mid \ell(\lambda(k)) = l \}| = \infty$,
        \item $S(k) \in \SYT(\lambda(k))$,
        \item when $S(k)$ is placed to the right of the standard skew Young tableau 
        \[ \biglsmash_{i=1}^{k-1} S(i),\] 
    the skew tableau resulting from the left smash of $S(k)$ is still standard. 
    In other words, each partial smash is a standard skew Young tableau.
    Note that for $k=1$, our adjusted instruction requires that $S(1)$ is placed within the shape $(\infty^l)$ and left justified before continuing with the left smash of $S(2)$.
    \end{itemize}
    We note that if $\sigma$ is a weakly decreasing skew shape with $S\in\SYT(\sigma)$, the choice of $\lambda(k) = \sigma$ and $S(k) = S$ for all $k\in \N$ satisfies the above condition for all $k\in \N$ provided that it is satisfied at $k=1$. Namely, all that is required is that left justification of $S$ results in a standard Young tableau. 
\end{remark}

\begin{remark}[Extension of left smash to infinite Young tableaux] \label{rem: smash variant tableau and infinite tableau}
    In Definition \ref{def: smash infty^l}, we may modify the construction to allow smashing between a standard Young tableau $S\in\SYT(\lambda)$ with $\lambda'_1\leq l$ and an infinite standard Young tableau $T\in\SYT((\infty^l))$, written
    \[ S \lsmash T.\]
    We also use the convention that $\varnothing \lsmash T = T$. 
\end{remark}

\begin{remark} \label{rem: uncountable SYT}
    By repeatedly smashing with two distinct standard Young tableaux from $\SYT((2^l))$ and using a standard binary coding argument, we see that $\SYT((\infty^l))$ and $\SYT((l^\infty))$ are uncountable for all $l\ge 2$ (while $\SYT((\infty^1))$ and $\SYT((1^\infty))$ are singletons).
\end{remark}

\subsection{The $\prom$ Algorithm}

There is a natural operation on tableaux using the process known as \emph{jeu de taquin} (introduced in \cite{Schutzenberger-Promotion-Original} and \cite{Schutzenberger1977} in the context of rectifying semistandard skew Young tableaux) that generalizes to infinite shapes. 
The procedure is the same, whether on a finite or infinite standard Young tableau, $T$.
The algorithm begins by deleting the box of $T$ with the entry $1$, leaving an empty space. Next, inductively iterate the following procedure: if there exists a box either directly to the right or directly below the empty space, slide the box with the smaller entry (which leaves a new empty space) into the position of the empty space. 
On a finite standard Young tableau, this procedure terminates at what we will shortly call an outer corner, resulting in a smaller tableau. 
On an infinite standard Young tableau, the procedure produces an infinite number of slides, and the resulting tableau has the same shape. 
Finally, subtract $1$ from each entry. The resulting tableau is denoted by 
\[ \prom(T). \]
This process preserves the property of being a standard Young tableau. An example in the infinite case is given in Figure \ref{fig:infinite promotion example}. 
We will also write $\prom^0(T) = T$ and, for a finite standard Young tableau, 
$\prom^{|T|}(T) = \varnothing$. Finally, if $S$ is a standard Young tableau or standard skew Young tableau contained in $T$, 
we write
\[ \prom_T(S) \]
for the subtableau or skew subtableau of $\prom(T)$ that results from tracing the movement of the boxes in $S$ when the operator $\prom$ is applied to $T$. 
\begin{figure}[H]
    \centering
    \begin{align*}
     \mathrel{\raisebox{-2ex}{$T$}} &\mathrel{\raisebox{-2ex}{=}} 
    \withlines{\begin{ytableau}
        1 & 4 & 7 & 10 & \none[\cdots]\\
        2 & 5 & 8 & 11 & \none[\cdots]\\
        3 & 6 & 9 & 12 & \none[\cdots]
    \end{ytableau}}
    &\mathrel{\raisebox{-2ex}{$\rightarrow$}}\;
    \withlines{\begin{ytableau}
         *(lightgray) 2 & 4 & 7 & 10 & \none[\cdots]\\
         *(lightgray)3 & 5 & 8 & 11 & \none[\cdots]\\
         *(lightgray)6 &  *(lightgray)9 &  *(lightgray)12 &  *(lightgray)15 & \none[\cdots]
    \end{ytableau}}
    &\mathrel{\raisebox{-2ex}{$\rightarrow$}}\;
    \withlines{\begin{ytableau}
        1 & 3 & 6 & 9 & \none[\cdots]\\
        2 & 4 & 7 & 10 & \none[\cdots]\\
        5 & 8 & 11 & 14 & \none[\cdots]
    \end{ytableau}}
    \mathrel{\raisebox{-2ex}{=}}& \mathrel{\raisebox{-2ex}{$\jdt(T)$}}
    \end{align*}
    \caption{$\jdt$ on an infinite standard Young tableau of shape $(\infty^3)$.}
    \label{fig:infinite promotion example}
\end{figure}


A notion of the jeu de taquin algorithm is also defined on standard skew Young tableaux. Let $\sigma = \lambda / \nu$ and $S\in\SYT(\sigma, A)$. An \emph{inner corner of $\sigma$} (not actually a box of $\sigma$) is any maximally southeastern box of $\nu$, meaning there are no boxes in $\nu$ to the right or below an inner corner. Note that such a box is also called an \emph{outer corner of $\nu$}.

Beginning with any inner corner of $\sigma$, viewed as the initial empty space, a \emph{jeu de taquin slide} or \emph{slide} is performed on $S$ by inductively iterating the following procedure: if there exists a box of $S$ either directly to the right or directly below the empty space, slide the box with the smaller entry (which leaves a new empty space) into the position of the empty space. This procedure will terminate after a finite number of steps and produce a standard (generally skew) Young tableau of the same size, but of a different shape. We write $\prom(S)$ for the resulting tableau and note that $\prom(S)$ depends on the initial choice of inner corner. 

By iteratively applying slides, it is possible to arrive at a standard Young tableau. 
This tableau is called the \emph{rectification of $S$} and two standard skew Young tableaux with the same rectification are called \emph{jeu de taquin equivalent}. It is well known that the rectification of $S$ is independent of the choice and order of slides and that jeu de taquin equivalence coincides with Knuth equivalence of the tableaux's reading words. As these facts are repeatedly required in Section \ref{sec: periodic}, we record them as a theorem for later use.

\begin{theorem}[\cite{fulton1997young}] \label{thm: Uniqueness of Rectified Skew Tableau}
    Let $\sigma$ be a skew shape and $S\in\SYT(\sigma, \N)$. The rectification of $S$ is independent of the choice and order of jeu de taquin slides, and two standard skew Young tableaux have the same rectification if and only if their reading words are Knuth equivalent.
\end{theorem}

We note that jeu de taquin slides and rectification do not include decrementing entries as a final step. When we wish to decrement with each jeu de taquin slide in a rectification, we will refer to the result as the \emph{decremented rectification}.

\section{Basic Results} \label{sec: basic}






For an infinite shape $\mu$, 
note that $\jdt : \SYT(\mu) \rightarrow \SYT(\mu)$ is surjective:
if $T \in \SYT(\mu)$ with $\mu =(\infty^l)$ or $(\infty^\infty)$, it is straightforward to verify that $U\in\prom^{-1}(T)$ where $U$ is defined by
    \[
    U_{i,j} = 
    \begin{cases}
        1, & \text{if } (i,j)=(1,1), \\     
        T_{1,j-1} + 1, & \text{if } i = 1 \text{ and } j \geq2, \\   
        T_{i, j} + 1 & \text{if } i \geq 2.
    \end{cases}
    \]
However, it is not injective as Figure \ref{fig: jdt not inj} shows. 
    \begin{figure}[H]
    \[
    S =
    \withlines{\begin{ytableau}
        1 & 2 & 3 & 5 & 7 & 9 & \none[\cdots]\\
        4 & 6 & 8 & 10 & 12 & 14 & \none[\cdots]
    \end{ytableau}} \hspace{.2cm}
    \hspace{.2cm}
    T =
    \withlines{\begin{ytableau}
        1 & 3 & 5 & 7 & 9 & 11 & \none[\cdots]\\
        2 & 4 & 6 & 8 & 10 & 12 & \none[\cdots]
    \end{ytableau}}
    \]

    \vspace{0.8em} 

    \[
    \prom(S) = \prom(T) =
    \withlines{\begin{ytableau}
        1 & 2 & 4 & 6 & 8 & 10 & \none[\cdots]\\
        3 & 5 & 7 & 9 & 11 & 13 & \none[\cdots]
    \end{ytableau}}
    \]
    \caption{Non-injectivity of $\jdt$.}
    \label{fig: jdt not inj}
\end{figure}  
We will see below, in Theorem \ref{Thm: Size of Preimage} and 
Remark \ref{rmk: Upper Bound is Sharp}, 
that the size of the preimage of a standard Young tableau under $\prom$ can be infinite for the infinite shape $(\infty^\infty)$, but is bounded by $l$ for the infinite shapes $(\infty^l)$ and $(l^\infty)$. Towards that end, we introduce the following definition from \cite{RomikSniady2015}.



\begin{definition} \label{Def: Path}
    Let $\mu$ be an infinite shape and $T \in \SYT(\mu)$.
    The \emph{sliding path of $T$} is the sequence of all the positions of $T$ where boxes got either deleted or moved during the $\prom$ algorithm. 
    
    More precisely, the sliding path   
    is the sequence of positions of $T$, 
    \[ \operatorname{p}(T) = (\operatorname{p}_1(T),\allowbreak \operatorname{p}_2(T),\allowbreak \operatorname{p}_3(T),\allowbreak  \ldots) \]
    recursively defined as follows:
    begin by setting $\operatorname{p}_1(T) = (1,1)$.     
    If the position $\operatorname{p}_n(T)$ is defined and there exists a box either directly to the right or directly below it, then 
    $\operatorname{p}_{n+1}(T)$ is the position of the box with the smaller entry.

    We say that the sliding path $\operatorname{p}(T)$ \emph{eventually becomes horizontal} if 
    $\operatorname{p}_n(T)$ remains on a fixed row for all sufficiently large $n\in\N$. In this case, we say that it \emph{becomes horizontal in row} $i_0$ if $\operatorname{p}_n(T)$ remains in row $i_0$ for all sufficiently large $n\in\N$. Additionally, we say that $\operatorname{p}(T)$ \emph{becomes horizontal at column} $j_0$ if the positions $\operatorname{p}_n(T)$ in columns $j$ with $j\geq j_0$ are all in a fixed row.
    
\end{definition}

See Figure \ref{fig: example prom path} for an example of Definition \ref{Def: Path}.
Also observe that, \[\mbox{if $\operatorname{p}_n(T) = (i,j)$, then $i+j=n+1$.}\]
In particular, for each $n\geq 2$, there is a unique box $(i,j)$ in the sliding path with $i+j=n$.
\begin{figure}[H]
    \centering
    \begin{tikzpicture}[remember picture]

      \node {
        \( \mathrel{\raisebox{-2ex}{$T=$}}
          \withlines{\begin{ytableau}
            \tikz[remember picture,baseline=(A.base)]\node[inner sep=2pt](A){1}; & 
            \tikz[remember picture,baseline=(B.base)]\node[inner sep=2pt](B){2}; & 
            \tikz[remember picture,baseline=(C.base)]\node[inner sep=2pt](C){4}; & 
            7 & 11 & \none \\
            3 & 5 & 
            \tikz[remember picture,baseline=(D.base)]\node[inner sep=2pt](D){6}; & 
            \tikz[remember picture,baseline=(E.base)]\node[inner sep=2pt](E){9}; & 
            14 & \none[\dots] \\
            8 & 10 & 12 & 
            \tikz[remember picture,baseline=(F.base)]\node[inner sep=2pt](F){13}; & 
            \tikz[remember picture,baseline=(G.base)]\node[inner sep=2pt](G){15}; & \none
          \end{ytableau}}
        \)
      };

      \begin{pgfonlayer}{background}
        \draw[line width=2pt, lightgray, -{latex}]
          (A.center) --
          (B.center) --
          (C.center) --
          (D.center) --
          (E.center) --
          (F.center) --
          (G.center) -- ([xshift=5pt]G.east);

        \fill[lightgray] (A.center) circle (3pt);
      \end{pgfonlayer}
    \end{tikzpicture}

    \caption{Sliding path for $T \in \SYT((\infty^3))$.}
    \label{fig: example prom path}
\end{figure}

We note a quick observation about sliding paths for the infinite shape $(\infty^l)$ that follows from the fact that, at each stage, sliding paths can only move to the right or down.

\begin{lemma} \label{lem: evenutal horizontal}
    For all $l\in\N$, $\mu = (\infty^l)$, and $T \in \SYT(\mu)$, the sliding path $\operatorname{p}(T)$ eventually becomes horizontal.
\end{lemma}

We begin with results on the behavior of elements of $\prom^{-1}$ and their sliding paths. The first follows immediately from the definition of $\prom$.

\begin{lemma} \label{lem: prom path dets inverse image equality}
    Let $\mu$ be an infinite shape and $T \in \SYT(\mu)$.
    If $U,V \in \prom^{-1}(T)$, then $U = V$ if and only if they have the same sliding paths.
\end{lemma}

\begin{lemma} \label{lem: Paths Never Cross}
    Let $\mu$ be an infinite shape and $T \in \SYT(\mu)$.
    If $U,V \in \prom^{-1}(T)$ 
    and $\operatorname{p}_n(U) \neq \operatorname{p}_n(V)$ for some $n\in\N$, then $\operatorname{p}_m(U) \neq \operatorname{p}_m(V)$ for every $m \geq n$.
\end{lemma}

\begin{proof}
    By way of contradiction, suppose there is a minimal $m > n$ such that $\operatorname{p}_m(U) = \operatorname{p}_m(V) = (i,j)$.
    Since $\operatorname{p}_{m-1}(U) \neq \operatorname{p}_{m-1}(V)$, we see that $\{\operatorname{p}_{m-1}(U), \operatorname{p}_{m-1}(V)\} = \{(i-1,j),(i,j-1)\}$.
    Possibly after relabeling, we may suppose $\operatorname{p}_{m-1}(U) = (i-1,j)$ and $\operatorname{p}_{m-1}(V) = (i,j-1)$.
    
    As $\prom(U) = T$ and $(i,j)$ is in the sliding path of $U$, it follows that
    $T_{i-1,j} = U_{i,j} - 1$.    
    As $\prom(U) = T$ and $(i,j-1)$ is not in the sliding path of $U$, it follows that 
    $T_{i,j-1} = U_{i,j-1} -1$. 
    Since $U_{i,j-1} < U_{i,j}$, it follows that 
    $T_{i,j-1} < T_{i-1,j}$. 
    Working with $V$ similarly, we see that
    $T_{i,j-1} = V_{i,j} - 1$ and $T_{i-1,j}=V_{i-1,j} -1$.
    Since $V_{i-1,j} < V_{i,j}$, it follows that
    $T_{i-1,j} < T_{i,j-1}$, a contradiction.
\end{proof}

It is actually possible to characterize the precise size of the preimage of an infinite tableau of shape $(\infty^l)$. It depends on the following definition and utilizes Lemma \ref{lem: Path to the Right}.

\begin{definition}\label{def: Dominated Entry}    
    Let $l\in\N$, $\mu=(\infty^l)$, and $T \in \SYT(\mu)$.
    An entry $T_{i,j}$ is \emph{dominated} if $T_{i,j} < T_{i-1,j+1}$.
\end{definition}


\begin{lemma}\label{lem: Path to the Right}
    Let $l,n\in\N$, $\mu=(\infty^l)$, $T \in \SYT(\mu)$, $U \in \prom^{-1}(T)$,
    and $\operatorname{p}_n(U) = (i,j)$.
    If $T_{i,j}$ is dominated, then $\operatorname{p}_{n+1}(U) = (i+1, j)$.
    Also, if $T_{i+1, j-1}$ is not dominated, then $\operatorname{p}_{n+1}(U) = (i, j+1)$.
\end{lemma}

\begin{proof}
    We prove the contrapositive of each statement.
    Since $\operatorname{p}_n(U) = (i,j)$, we know that $(i-1,j+1)$ is not in the sliding path of $U$, and so $T_{i-1,j+1} = U_{i-1,j+1} - 1$.
    Thus, if $\operatorname{p}_{n+1}(U) = (i,j+1)$, then $T_{i,j} = U_{i,j+1} - 1 > U_{i-1,j+1} - 1 = T_{i-1,j+1}$.
    Hence $T_{i,j}$ is not a dominated entry.

    Similarly, if $\operatorname{p}_{n+1}(U) = (i+1,j)$, then $T_{i,j} = U_{i+1,j} - 1 > U_{i+1,j-1} - 1 = T_{i+1,j-1}$.
    Hence $T_{i+1,j-1}$ is a dominated entry.
\end{proof}

The proof of Theorem \ref{Thm: Size of Preimage} explicitly constructs each element in the preimage and its sliding path.

\begin{theorem} \label{Thm: Size of Preimage}
    Let $l\in\N$, $\mu=(\infty^l)$, and $T \in \SYT(\mu)$.
    Then 
    \[|\prom^{-1}(T)| = l - k,\]
    where $k$ is the number of rows of $T$ that contain infinitely many dominated entries.
\end{theorem}

\begin{proof}
    From Lemma \ref{lem: prom path dets inverse image equality}, the sliding path uniquely determines elements of $\prom^{-1}(T)$. From Lemma \ref{lem: Paths Never Cross}, we see that the sliding paths of distinct elements of $\prom^{-1}(T)$ eventually diverge. 
    As sliding paths of elements of $\SYT(\mu)$ must eventually become horizontal (Lemma \ref{lem: evenutal horizontal}), sliding paths of distinct elements of $\prom^{-1}(T)$ eventually must remain on distinct rows of $\mu$. By Lemma \ref{lem: Path to the Right}, if a row of $T$ contains infinitely many dominated entries, then no sliding path of any of its preimages can eventually become horizontal in that row.
    Consequently, we get $|\prom^{-1}(T)| \leq l - k$.


    For $|\prom^{-1}(T)| \geq l - k$, suppose that $T$ contains at most finitely many dominated entries in row $i_0$. We construct a path $\operatorname{q} = (\operatorname{q}_1, \operatorname{q}_2, \dots)$ with $q_1=(1,1)$ that becomes horizontal in row $i_0$ and use it to construct $U\in\prom^{-1}(T)$ with $\operatorname{p}(U)=\operatorname{q}$.
    
    Let $j_0$ be minimal so that $T_{i_0,j}$ is not dominated for any $j\geq j_0$. For all $n\geq i_0+j_0-1$, define
    \[ \operatorname{q}_n = (i_0, n + 1 - i_0)\]
    so that $\operatorname{q}_n$ lies on row $i_0$ for $n\geq i_0+j_0-1$.
    Next, we iteratively construct $\operatorname{q}_n$ backwards down to $n = 1$. If $\operatorname{q}_n$ is defined for $n\geq r+1$ 
    with $r\geq 1$ and $\operatorname{q}_{r+1}=(i,j)$, define $\operatorname{q}_{r}$ as follows:
    \[
    \operatorname{q}_{r} =
    \begin{cases}
        (i-1,j), \text{ if $j = 1$ or $T_{i,j-1}$ is dominated}, \\
        (i,j-1), \text{ if $i = 1$ or $T_{i,j-1}$ is not dominated}. 
    \end{cases}
    \]
    It follows from this construction that $q$ defines a path of neighboring positions in $T$ and that for all $n\in \N$ if $\operatorname{q}_n = (i,j)$, then $i + j = n+ 1$. 
    
    Now, we define an infinite tableau $U$ of shape $\mu$ as follows:
    \begin{itemize}
        \item If $(i,j) \neq \operatorname{q}_n$  for all $n \in \N$, let $U_{i,j} = T_{i,j} + 1$.

        \item Let $U_{\operatorname{q}_1} = 1$ and $U_{\operatorname{q}_n} = T_{\operatorname{q}_{n-1}} + 1$ for all $n \geq 2$.
    \end{itemize}
    Note that this defines a bijective $\N$-filling of the boxes of $\mu$.
    
    To see that $U \in \SYT(\mu)$, we show that $U_{i,j} > U_{i-1,j}$ and $U_{i,j} > U_{i,j-1}$ for all $i,j\ge 1$ (where we interpret $U_{i,j} = T_{i,j} = 0$ if $i \leq 0$ or $j \leq 0$).
    Fix $i,j \in \N$.
    By construction, 
    \[U_{i,j} \in \{T_{i,j} + 1, T_{i-1,j} + 1, T_{i,j-1} + 1\},\]
    and similarly
    \begin{eqnarray*}
    & & U_{i-1,j} \in \{T_{i-1,j} + 1, T_{i-2,j} + 1, T_{i-1,j-1} + 1\} ~ \text{and}\\
    & & U_{i,j-1} \in \{T_{i,j-1} + 1, T_{i-1,j-1} + 1, T_{i,j-2} + 1\}.
    \end{eqnarray*}
    We show that the claim holds for each possible option for $U_{i,j}$ as listed above.
    \begin{enumerate}
        \item Suppose $U_{i,j} = T_{i,j} + 1$.
        Then, either $i = 1$ (and $U_{i,j} > U_{i-1,j}$ holds trivially) or $U_{i,j} > U_{i-1,j}$ since $T \in \SYT(\mu)$.
        Similarly, either $j = 1$ or $U_{i,j} > U_{i,j-1}$ since $T \in \SYT(\mu)$.

        \item If $U_{i,j} = T_{i-1,j} + 1$, then $\operatorname{q}_{i + j - 1} = (i,j)$, $\operatorname{q}_{i + j - 2} = (i-1,j)$, and $T_{i,j-1}$ is dominated or $j = 1$.
        Observe that $U_{i-1,j} \neq U_{i,j}= T_{i-1,j} + 1$ by construction of $U$.
        Thus $U_{i,j} > U_{i-1,j}$ since $T \in \SYT(\mu)$.
        
        If $j = 1$, then $U_{i,j} > U_{i,j-1}$ holds trivially. 
        We assume that $j > 1$, and so $T_{i,j-1}$ is dominated.
        Observe that $U_{i,j-1} \neq T_{i,j-2} + 1$ and $U_{i,j-1} \neq T_{i-1,j-1} + 1$, since otherwise we obtain the contradiction $(i,j-1) = \operatorname{q}_{i + j - 2} = (i - 1,j)$.
        Thus $U_{i,j-1} = T_{i,j-1} + 1$, and so $U_{i,j} = T_{i-1,j} + 1 > T_{i,j-1} + 1 = U_{i,j-1}$ since $T_{i,j-1}$ is dominated.

        \item If $U_{i,j} = T_{i,j-1} +1$, then $\operatorname{q}_{i + j - 1} = (i,j)$ and $\operatorname{q}_{i + j - 2} = (i,j-1)$.
        The argument is similar to Case (2).
    \end{enumerate}
    It only remains to show that $\operatorname{p}_n(U) = \operatorname{q}_n$ for each $n \in \N$ since, by the construction of $U$, it will then follow immediately that $\prom(U) = T$.
    We do this by induction on $n \in \N$.
    Clearly $\operatorname{q}_1 = (1,1) = \operatorname{p}_1(U)$, so suppose $\operatorname{q}_n = (i,j) = \operatorname{p}_n(U)$ for some $n \geq 1$.
    Then $\operatorname{q}_{n + 1} \in \{(i+1,j),(i,j+1)\}$.
    If $\operatorname{q}_{n+1} = (i+1,j)$, then $(i,j+1) \neq \operatorname{q}_m$ for any $m \in \N$. 
    Hence
    \[U_{i+1,j} = T_{i,j} + 1 < T_{i,j+1} + 1 = U_{i,j+1}\]
    because $T \in \SYT(\mu)$, and so $\operatorname{p}_{n+1}(U) = (i+1,j) = \operatorname{q}_{n+1}$.
    If $\operatorname{q}_{n+1} = (i,j+1)$, the argument is similar and omitted. 
    %
    %
    %
\end{proof}

\begin{remark} \label{rmk: Upper Bound is Sharp}
    In particular, Theorem \ref{Thm: Size of Preimage} shows that $|\jdt^{-1}(T)| \leq l$ for $T \in \SYT((\infty^l))$.
    It is straightforward to show that this bound is sharp by explicitly constructing $T(l)\in\SYT((\infty^l))$ with $|\prom^{-1}(T(l))| = l$. To do so,  fill boxes diagonally (along diagonals that correspond to positions $(i,j)$ with $i+j$ constant) 
    with the smallest available entry, working down from the top. 
    For example, $T(3)$ is pictured in Figure \ref{fig: example prom^-1 = l}.
    \begin{figure}[H]
        \[ \mathrel{\raisebox{-2ex}{$T(3)$ =}} 
        \withlines{\begin{ytableau}
            1 & 2 & 4 & 7 & 10 & 13 & \none[\cdots]\\
            3 & 5 & 8 & 11 & 14 & 17 &\none[\cdots]\\
            6 & 9 & 12 & 15 & 18 & 21 & \none[\cdots]
        \end{ytableau}} \hspace{.2cm}
        \]
        \caption{$T(3) \in \SYT((\infty^3))$ with $|\prom^{-1}(T(3))| = 3$.}
        \label{fig: example prom^-1 = l}
    \end{figure}
    It is straightforward to verify in this case that $|\prom^{-1}(T(l))| = l$ and that the $i$th corresponding sliding path moves down the first column to row $i$ and then becomes horizontal.
    
    Based on Theorem \ref{Thm: Size of Preimage}, it is natural to guess that there exist $T\in\SYT((\infty^\infty))$ with $|\prom^{-1}(T)| = \infty$. In fact, this is true and the natural diagonal construction above provides an example. 
    In Section \ref{sec: concluding remarks}, we outline a construction showing that the preimage $\prom^{-1}(T)$ can be uncountable.
\end{remark}




We end this section with an observation on the interaction of left smashing, see Remark~\ref{rem: smash variant tableau and infinite tableau}, and the $\prom$ algorithm.

\begin{lemma} \label{lem: Non-rectangluar Smashing}
    Let $n,l\in\N$, $\lambda\vdash n$ with $\lambda'_1 \leq l$, $\mu=(\infty^l)$, $S\in\SYT(\lambda)$, and $T\in\SYT(\mu)$.
    Then
    \[ 
    \prom\left(S \lsmash T\right) = \jdt(S) \lsmash T.
    \]
\end{lemma}

\begin{proof}
    Observe that, by construction, the sliding path of $S \lsmash T$ intersected with $S$ is identical to the sliding path of $S$, and that the final box of the sliding path in $S$, $(i_0,\lambda_{i_0})$, is an outer corner of $\lambda$. 
    We claim that the sliding path of $S \lsmash T$ becomes horizontal in row $i_0$.
    If so, the result follows as $\prom$ performs its slides on $S$ and then shifts a row of $T$ left to replace the vacated box $(i_0,\lambda_{i_0})$ of $S$.

    If $i_0 = l$, there is nothing to check. Otherwise, consider $1 \leq i_0 < l$ and $j\geq \lambda_{i_0} + 1$. 
    In that case, the construction of $S \lsmash T$ shows that 
    \[ (S \lsmash T)_{i_0, j} = T_{i_0, j - \lambda_{i_0}} + n \text{ and} \]
    \[ (S \lsmash T)_{i_0+1, j - 1} = T_{i_0 +1, j -1 - \lambda_{i_0+1}} + n.\]
    Moreover, we have $\lambda_{i_0} \geq \lambda_{i_0+1} + 1$ since $(i_0,\lambda_{i_0})$ is an outer corner of $\lambda$, so that 
    \[ T_{i_0, j - \lambda_{i_0}} < T_{i_0+1, j - \lambda_{i_0}} \leq T_{i_0 +1, j -1 - \lambda_{i_0+1}}. \]
    Hence 
    \[(S \lsmash T)_{i_0,j}  < (S \lsmash T)_{i_0 +1,j - 1}.\]
    It follows by induction that $\operatorname{p}_n(S \lsmash T) = (i_0, j)$ for all $n = i_0 + j - 1 \geq i_0 + \lambda_{i_0} - 1$.
\end{proof}

\section{Periodic Behavior of $\prom$} \label{sec: periodic}


\begin{definition} \label{Def: Periodic}
    Let $\mu$ be an infinite shape and $T\in\SYT(\mu)$. We say that $T$ is \emph{periodic} with \emph{period $k$}, also written \emph{$k$-periodic}, if $\prom^k(T) = T$. 
\end{definition}

\begin{theorem} {\label{thm: Periodic gives linear growth}}
    Let $\mu$ be an infinite shape and $T\in\SYT(\mu)$.
    If $T$ is $k$-periodic, then $T_{i,j}\leq k(i+j-2)+1$.
    Moreover, if $l\in\N$ and $\mu = (\infty^l)$, then $k\geq l$.
\end{theorem}

\begin{proof}    
    As $\prom^k(T) = T$ and boxes that do not slide under $\prom^k$ are simply decremented by $k$,
    each box $(i,j)$ must slide under each iteration of $\prom^k$. It follows that there is an $m\in\N$ so that the box $(i,j)$ slides to the box $(1,1)$ under $\prom^m$, and as every path from $(i,j)$ to $(1,1)$ that moves only up or to the left has length $i+j-2$, we see that $m\leq k(i+j-2)$. Finally, if the box $(i,j)$ slides to the box $(1,1)$ under $\prom^m$, it follows that $1=T_{i,j} -m$, so that $T_{i,j} = m+1 \leq k(i+j-2) + 1$.

    For the final statement, recall from Lemma \ref{lem: evenutal horizontal} that for each $0 \leq j \leq k-1$ the sliding path of the standard Young tableau $\prom^j(T)$ eventually becomes horizontal. As every box must slide under $\prom^k$, these horizontal lines must include every row of $\mu$. Therefore $k\geq l$.    
\end{proof}



\begin{example} 
    In light of Theorem \ref{thm: Periodic gives linear growth}, one might hope that linear growth of the entries of $T \in \SYT(\mu)$ forces periodicity. However, this is not true.
    For example, let
        \[Y = \ytableaushort{12,3} \,\,\,\text{ and }\,\,\, Z = \ytableaushort{13,2}\]
    and define $T\in\SYT((\infty^2))$ by smashing an increasing number of copies of $Y$ before each copy of $Z$:
        \[ T = Y\lsmash Z \lsmash Y\lsmash Y\lsmash Z \lsmash Y\lsmash Y \lsmash Y\lsmash Z \lsmash \ldots. \]
    It is straightforward to show that $T_{i, j} \leq 3j$, but that $T$ is not periodic.
    (In fact, with Theorem \ref{thm: backfill and alt periodic classification}, every row of a periodic tableau $T$ must eventually show a periodic pattern.)
\end{example}

As a corollary of Theorem \ref{thm: Periodic gives linear growth}, we get the following.

\begin{corollary} \label{Cor: doubly infinite is never periodic}
    No element of $\SYT{((\infty^\infty))}$ is periodic.
\end{corollary}

\begin{proof}
    By the definition of standard Young tableau, it follows that any $T\in\SYT{((\infty^\infty))}$ satisfies $T_{i, j}\geq ij$. In particular, $T_{i,i}$ has at least quadratic growth in $i$ so that Theorem \ref{thm: Periodic gives linear growth} finishes the argument.
\end{proof}

We come now to a structure theorem for periodic infinite tableaux in $\SYT((\infty^l))$.

\begin{theorem}\label{thm: Structure Thm}
    Let $l,k\in\N$ with $k\geq l$, $\mu=(\infty^l)$, and $T\in \SYT(\mu)$.
    Then 
    \[ \prom^k(T)=T \]
    if and only if there exists $p_0 \in \N$ and, for $1\leq i \leq p_0+1$,
    there are skew shapes $\sigma(i)$ contained in $\mu$ and $S(i)\in\SYT(\sigma(i), \N)$ 
    such that
    \begin{itemize}
        \item $S(i) = T|_{\sigma(i)} = T[1+(i-1)k, ik]$\,\, for all\,\,  $1\leq i \leq p_0+1$,
        \item $\prom^k_T(S(i)) = S(i-1)$\,\, for all\,\, $2\leq i \leq p_0+1$,
        \item the shifts $\tilde{S}(i) := S(i) - (i-1)k$ satisfy that $\tilde{S}(p_0)$ and $\tilde{S}(p_0+1)$ are row equivalent, and
        \item $T = \tilde{S}(1) \lsmash \tilde{S}(2) \lsmash \dots \lsmash \tilde{S}(p_0 -1) \lsmash \biglsmash_{i=p_0}^\infty \tilde{S}(p_0)$.
    \end{itemize}
    In this case, $l =\ell( \sigma(p_0))  \geq \ell(\sigma(p_0-1)) \geq \ldots \geq \ell(\sigma(1))$,
    $\sigma(1)$ is a Young diagram, and $\sigma(p_0)$ is weakly decreasing.
\end{theorem}

\begin{proof}
    For $i\in\N$, let
    \[ S(i) := T[1+(i-1)k, ik] \]
    and write $\sigma(i)$ for the corresponding skew shape (a Young diagram for $i=1$).
    By definition, we see that $T$ is $k$-periodic if and only if
    $\prom_T^k(S(i))=S(i-1)$
    for all $i\geq 2$.
    By Lemma~\ref{lem: evenutal horizontal}, there is a $j_0\in\N$ so that the sliding paths of the infinite Young tableaux $T, \prom(T), \prom^2(T),\ldots,\prom^{k-1}(T)$ all become horizontal at column $j_0$.

    If $T$ is $k$-periodic, choose $p_0\in\N$ minimal so that no box of $S(p_0)$ is to the left of column $j_0$, and let $i\geq p_0 +1$. By $k$-periodicity and horizontal sliding paths, it follows that $\prom^k$ decrements by $k$ and slides $S(i)$ horizontally onto $S(i-1)$, meaning that $S(i)$ and $S(i-1)$ have identical row sizes with corresponding entries differing by $k$, hence $\tilde{S}(i)\sim \tilde{S}(i-1)$. The result follows.

    Conversely, if $T$ satisfies the conditions, it remains to show that the successive smashes of each copy of $\tilde{S}(p_0)$ map to each other under $\prom^k$. First, note that successively smashing with $\tilde{S}(p_0)$ can only fill all of $\mu$ if $\ell( \sigma(p_0))=l$. Moreover, successively smashing with $\tilde{S}(p_0)$ means that column entries corresponding to shorter rows of $\tilde{S}(p_0)$ will eventually outpace those entries corresponding to longer rows of $\tilde{S}(p_0)$.    
    Thus, successive smashing with $\tilde{S}(p_0)$ can only lead to a tableau $T$ with all its columns increasing if the skew shape $\sigma(p_0)$ is weakly decreasing. Next, $\tilde{S}(p_0+1) \sim \tilde{S}(p_0)$ implies that no box has moved up while applying $\prom^k$ to $S(p_0+1)$. As in the proof of Lemma \ref{lem: Non-rectangluar Smashing}, from this and the fact that $\sigma(p_0)$ is weakly decreasing it follows that no box moves up while applying $\prom^k$ to any of the later smashed copies of $\tilde{S}(p_0)$, and the result follows.
    
    The final string of inequalities is immediate from the definition of $\prom$.
\end{proof}

\begin{remark}\label{rem: weak extension}
    It immediately follows from either Lemma \ref{lem: Non-rectangluar Smashing} or Theorem \ref{thm: Structure Thm} that any standard Young tableau can be extended to a periodic infinite standard Young tableau.
    More precisely, if $n \in \N$, $\lambda \vdash n$, and $R \in \SYT(\lambda)$, then for every $l \geq \lambda'_1$ there is some $T \in \SYT((\infty^l))$ such that $T|_\lambda = R$ and $\jdt^k(T) = T$ for some $k \in \N$.
    In particular, we can construct such a tableau $T$ by first extending $R$ to some standard Young tableau $\tilde{R} \in \SYT(\nu)$ where $\nu$ contains $\lambda$ and $\nu_1' = l$.
    Then, we may take $T := \biglsmash_{i=1}^\infty \tilde{R}$ and $k := |\tilde{R}|$.
    

    In fact, as we will show in Corollary \ref{cor: fine extension}, we can extend $R$ to an $|R|$-periodic $T \in \SYT((\infty^l))$ for any $\lambda_1' \leq l \leq |R|=n$.
\end{remark}

Our next lemma provides a general construction for periodic (and pre-periodic) infinite standard Young tableaux and demonstrates the use of rectification and reading words.

 \begin{lemma} \label{lem: One Lemma to Rule Them All}
    Let $l,n\in\N$, $\mu=(\infty^l)$, $\sigma = \nu_1 / \nu_2$ a weakly decreasing skew shape contained in $\mu$ with $\ell(\sigma)=l$, and $S\in\SYT(\sigma)$. Write 
    $k=|S|$. Now take $\lambda\vdash n$ contained in $\mu$ and $B\in\SYT(\lambda)$ such that $B \lsmash S$ is a standard Young tableau. (In many cases, we may have the natural choice $\lambda = \nu_2$.)
    If $T$ is defined as 
    \[ T = B \lsmash \left(\biglsmash_{i = 1}^\infty S \right), \]
    then $T\in\SYT(\mu)$ and
    \[ \tilde{T}=\prom^{n}(T)\]
    is $k$-periodic.
    Moreover, $\tilde{T}$ is of the form 
    \[ \tilde{T}=\tilde{B} \lsmash \left(\biglsmash_{i = 1}^\infty S \right)\]
    with $\tilde{B}$ the decremented rectification of $T[n+1, n+ (i_0-1) k]$ for some $i_0\in\N$ (i.e., a smash of $i_0-1$ copies of $S$).
\end{lemma}

\begin{proof}
    For $T\in\SYT(\mu)$, see Definition \ref{def: smash infty^l}: the condition $\ell(\sigma)=l$ ensures that the left smash results in a complete filling of all of $\mu$, and filling in a sequence of copies of $S\in\SYT(\sigma)$ for a weakly decreasing skew shape $\sigma$ with  $B \lsmash S$ a standard Young tableau ensures that the smashing results in a standard Young tableau $T$ with increasing columns.
    
    Let $j\geq 2$. By construction, $T[n+1,n+(j-1)k]$ and $T[n+k+1,n+jk]$ are both smashes of $j-1$ copies of $S$ 
    (with the first copy of $S$ fixed in place, respectively). The first set of $S$'s are incremented by $n, n+k, \ldots, n + (j-2)k$, respectively, and the second set of $S$'s are incremented by $n+k, n+2k, \ldots, n + (j-1)k$, respectively. As a result, the reading word for $T[n+k+1,n+jk]$ is obtained from the reading word of $T[n+1,n+(j-1)k]$ by incrementing each letter by $k$. 
    It then follows from Theorem \ref{thm: Uniqueness of Rectified Skew Tableau} that 
    the rectification of $T[n+k+1,n+jk]$ is obtained from the rectification of $T[n+1,n+(j-1)k]$ by incrementing each letter by $k$. 
    However, as the definition of the $\prom$ algorithm decrements entries, we see that $\prom_T^n$ implements the decremented rectification of $T[n+1,n+(j-1)k]$ and $\prom_T^{n+k}$ implements the decremented rectification of $T[n+k+1,n+jk]$, and that they are identical.

    Hence, as
    \[ \prom_T^n(T[n+1,n+(j-1)k]) = \tilde{T}[1, (j-1)k],\] 
    we get that 
    \[ \prom_T^{n+k}(T[n+k+1,n+jk]) =  \tilde{T}[1, (j-1)k]\] 
    as well.
    But also observe that
    \[ \prom_T^{n+k}(T[n+k+1,n+jk]) = \prom_{\tilde{T}}^k( \tilde{T}[k+1, jk] ).\]
    We therefore have that
    \[ \prom_{\tilde{T}}^k( \tilde{T}[k+1, jk] ) = \tilde{T}[1, (j-1)k] \]
    for each $j \geq 2$.    
    Using this equation and writing $\tilde{T}(i):=\tilde{T}[1 +(i-1)k, ik]$ for $i\geq 1$ and $\tilde{T}(0) := \varnothing$, it follows from induction that
    \[ \prom_{\tilde{T}}^k(\tilde{T}(i)) = \tilde{T}(i-1) \mbox{ for all $i\ge 1$},\]
    and periodicity follows.

    By Lemma \ref{lem: evenutal horizontal}, there exists $j_0\in\N$ so that the sliding paths of the infinite Young tableaux $T, \prom(T),\ldots, \prom^{n-1}(T)$ all become horizontal at column $j_0$. Let $i_0\in\N$ be such that no box of $T[n+1+(i_0-1)k,n+i_0k]$ is to the left of column $n+j_0$.  Due to the horizontal sliding paths, it follows that no box moves up while applying $\prom^n$ to this or any of the later smashed copies of $S$. In particular, with  $\prom_T^n(T[n+1+(i-1)k,n+ik]) = \tilde{T}(i)$,
    the shift $\tilde{T}(i) - (i-1)k$ is row equivalent to $S$ for all $i \ge i_0$. It is now easy to check that $\tilde{T}$ is the smash of $\tilde{B}$ with infinitely many copies of $S$, where $\tilde{B}$ denotes the decremented rectification of $T[n+1, n+ (i_0-1) k]$.
\end{proof}

Using Lemma \ref{lem: One Lemma to Rule Them All}, we associate to a standard skew Young tableau the rectification of a certain smash of copies of that tableau.

\begin{definition} \label{def: canon rect of smash of copies of S}
    Let $\sigma = \nu_1 / \nu_2$ be a weakly decreasing skew shape,
    $k=|\sigma|$, $l=\ell(\sigma)$, $n=|\nu_2|$, $\mu=(\infty^l)$, $S\in\SYT(\sigma)$, and $B\in\SYT(\nu_2)$. Define $T\in\SYT(\mu)$ by
    \[  T = B \lsmash \left( \biglsmash_{i=1}^\infty S \right). \]
    From the proof of Lemma \ref{lem: One Lemma to Rule Them All}, choose $j_0\in\N$ minimal so that the sliding paths of the infinite Young tableaux $T, \prom(T),\ldots, \prom^{n-1}(T)$ all become horizontal at column $j_0$, and choose $i_0\in\N$ minimal so that no box moves up while applying $\prom^n$ to $T[n+1+(i_0-1)k,n+i_0k]$.
    Define $R(S)$ to be the decremented rectification of $T[n+1, n+ (i_0-1) k]$ (which is a smash of $i_0-1$ copies of $S$ with the first copy of $S$ fixed in place), i.e.,
    \[ R(S) = \prom_T^n(T[n+1, n+ (i_0-1) k]). \]
    Note that this construction is independent of the particular choice of $B\in \SYT(\nu_2)$ by 
    Theorem \ref{thm: Uniqueness of Rectified Skew Tableau}. 
    Also note, by the definition of left smash, that $R(S)$ depends on $S$ only up to row equivalence of $S$.
\end{definition}

We come now to an alternate characterization of periodic tableaux in $\SYT((\infty^l))$ and use the notation $R(S)$ as above.

\begin{theorem}\label{thm: backfill and alt periodic classification}
    Let $l,k\in\N$, $\mu=(\infty^l)$, and $T\in\SYT(\mu)$. 
    Then $T$ is $k$-periodic if and only if $T$ is of the form     
    \[ T = R(S) \lsmash \left( \biglsmash_{i=1}^\infty S \right) \]
    for some $S\in\SYT(\sigma)$ with $\sigma$ a weakly decreasing skew shape contained in $\mu$ satisfying $\ell(\sigma)=l$ and $|\sigma|=k$.
\end{theorem}

\begin{proof}
    Lemma \ref{lem: One Lemma to Rule Them All} shows that all $T$ as constructed above are $k$-periodic.

    Now suppose $T$ is $k$-periodic. From Theorem \ref{thm: Structure Thm}, $T$ is the left smash of a standard Young tableau $R$ with $k\mid |R|$ and an infinite number of copies of a standard skew Young tableau $S$ with $|S|=k$. Finish the proof by applying $\prom^{|R|}$ to $T$, using Lemma \ref{lem: One Lemma to Rule Them All}.
\end{proof}

As already observed in Remark \ref{rem: weak extension}, any standard Young tableau $R \in \SYT(\lambda)$ with $\lambda'_1 \le l$ can be extended to a periodic infinite standard Young tableau $T \in \SYT((\infty^l))$. However, the periodicity of $T$ may be larger than the size of the tableau $|R|$. Below, in Corollary~\ref{cor: fine extension}, we show that, in fact, for $\lambda_1' \leq l \leq |R|$ it is possible to obtain an extension with period $|R|$. We begin with a lemma.

\begin{lemma}\label{lem: un-rectify to full height}
    Let $n,l \in \N$, $\lambda \vdash n$ with $\lambda_1'\leq l\leq n$, and $R \in \SYT(\lambda)$.
    Then there is a weakly decreasing skew shape $\sigma$ with $\ell(\sigma) = l$ and $S \in \SYT(\sigma)$ such that $R$ is the rectification of~$S$.
\end{lemma}

\begin{proof} The standard skew Young tableau $S \in \SYT(\sigma)$ results from applying the following construction steps to $R \in \SYT(\lambda)$.

\underline{Step 1:} Construct the standard skew Young tableau $R'$ by sliding each row of $R$ just enough to the right so that no two boxes are in the same column. More precisely, let $R' \sim R$ such that, for each $1\le j\le n$, $R'$ has exactly one box in column $j$, see Figure \ref{subfig:4.10b}.

\underline{Step 2:}
To construct $S$ from $R'$, repeatedly increase the height of the current Young diagram as needed by performing the following two operations: first, lower every row that contains only a single box by one level; second, locate the lowest row that contains more than one box and lower its first (leftmost) box to the row directly below. 
Repeat until you reach a skew shape $\sigma$ with $\ell(\sigma) = l$ and the skew tableau $S \in \SYT(\sigma)$, see Figure \ref{subfig:4.10c}.

By construction, $\sigma$ is weakly decreasing, and it is easy to check that jeu de taquin slides lead from $S$ back to $R'$ and from $R'$ back to $R$ as the rectification of $S$. 
\end{proof}    

\begin{figure}[H]
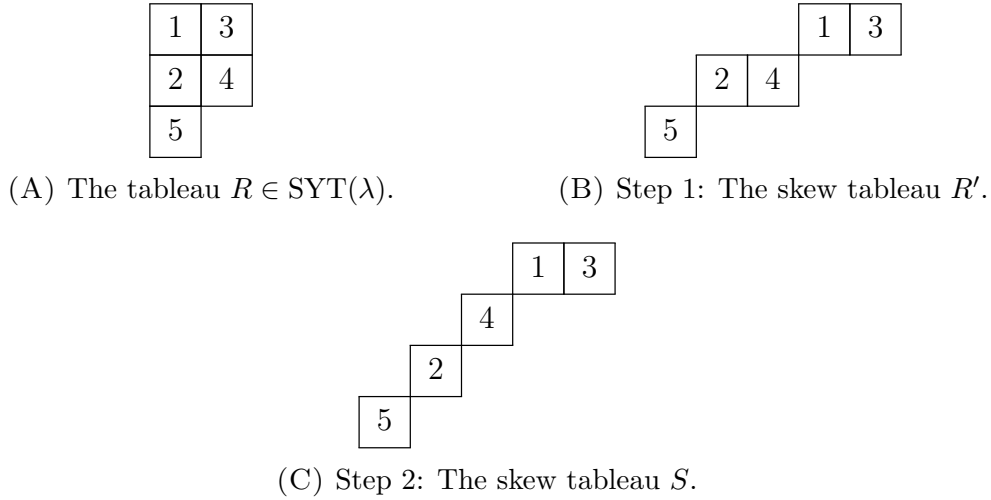

     \centering
     \ytableausetup{boxsize=1.6em} 
     \begin{subfigure}[b]{0.4\textwidth}
         \centering
         \ytableaushort{13,24,5}
         \caption{The tableau $R \in \SYT(\lambda)$.}
         \label{subfig:4.10a}
     \end{subfigure}
     \qquad
     \begin{subfigure}[b]{0.4\textwidth}
         \centering
        \begin{ytableau}
            \none & \none & \none & 1 & 3 \\
            \none & 2 & 4 \\
            5
        \end{ytableau}
         \caption{Step 1: The skew tableau $R'$.}
         \label{subfig:4.10b}
     \end{subfigure}
     
\vspace{1em}

     \begin{subfigure}[b]{0.4\textwidth}
         \centering
        \begin{ytableau}
            \none & \none & \none & 1 & 3 \\
            \none & \none & 4 \\
            \none & 2 \\
            5
        \end{ytableau}
         \caption{Step 2: The skew tableau $S$.}
         \label{subfig:4.10c}
     \end{subfigure}
     \caption{Example where $n=5$, $l=4$, and $\lambda = (2,2,1)$.}
     \label{fig:Lemma4.10}
\end{figure}

We immediately get the following periodic extension result.

\begin{corollary}\label{cor: fine extension}
    Let $n,l \in \N$, $\lambda \vdash n$ with $\lambda_1'\leq l\leq n$, $\mu = (\infty^l)$, and $R \in \SYT(\lambda)$.
    Then there is some $T \in \SYT(\mu)$ such that $T|_\lambda = R$ and $\jdt^n(T) = T$.
\end{corollary}

\begin{proof}
Apply Lemma \ref{lem: un-rectify to full height} to $R$ for some $S\in\SYT(\sigma)$ with $\sigma$ a weakly decreasing skew shape with $\ell(\sigma)=l$ and $|\sigma|=n$. Next apply Theorem \ref{thm: backfill and alt periodic classification} to $S$ to get a suitable $T \in \SYT(\mu)$.
\end{proof}

We end with a result on pre-periodic behavior.

\begin{definition} \label{def: pre-periodic}
    Let $\mu$ be an infinite shape and $T\in\SYT(\mu)$. We say that $T$ is \emph{pre-periodic} with \emph{eventual period} $k$, also written \emph{$k$-pre-periodic}, if $\prom^N(T)$ is $k$-periodic for some $N\in\N$.
\end{definition}

The following can be thought of as a partial converse of Lemma \ref{lem: One Lemma to Rule Them All} and gives a characterization of pre-periodic infinite standard Young tableaux.

\begin{theorem} \label{Thm: Characterization of Pre-periodic}
    Let $l,k \in \N$, $\mu = (\infty^l)$, and $T \in \SYT(\mu)$. 
    Then $T$ is $k$-pre-periodic if and only if there is
    a weakly decreasing skew shape $\sigma$ contained in $\mu$ with $\ell(\sigma)=l$ and $|\sigma|=k$, $S\in\SYT(\sigma)$, $n\in\N$, $\lambda\vdash n$ contained in $\mu$, and $B\in\SYT(\lambda)$ such that $B \lsmash S$ is a standard Young tableau and 
    \[T = B \lsmash \left( \biglsmash_{i=1}^\infty S \right).\]
\end{theorem}

\begin{proof}
    Lemma \ref{lem: One Lemma to Rule Them All} takes care of the ``if'' direction of the proof. For the ``only if'' direction, suppose $\prom^N(T)$ is periodic with period $k$. Then there are at most $N+k$ distinct sliding paths corresponding to the orbit of $T$ under $\prom$. By Lemma \ref{lem: evenutal horizontal}, choose $j_0 \in\N$ so that each of these sliding paths becomes horizontal at column $j_0$. 
    
    For $i\in \N$, write \[T(i):=T[N+1+(i-1)k, N +ik] \quad \mbox{ and } \quad
    T^N(i):=\prom^N(T)[1+(i-1)k,ik].\]
    Let $i_0\in \N$ be such that no box of $T^N(i_0)$ is to the left of column $j_0$, and let $i\ge i_0+1$.
    By $k$-periodicity and horizontal sliding paths, applying $\prom^k$ to $\prom^N(T)$ horizontally slides $T^N(i)$ to $T^N(i-1)$ with a decrement of $k$. Moreover, by the horizontal sliding paths, $T^N(i-1)$ is row equivalent to $T(i-1)-N$.
    Thus, with $B:=T[1,N+(i_0-1)k]$ and $S:=T(i_0)-(N+ (i_0-1)k)$, we have
    \[T = B \lsmash \left( \biglsmash_{i=1}^\infty S \right).\]
    Moreover, the shape of $S$ must be weakly decreasing or else the columns of $T$ would eventually not be increasing as the entries in the smash of any higher, shorter row would eventually become larger than the entries in any lower, longer row.
    Finally, by construction, $B\lsmash S$ is standard.    
\end{proof}


\section{Chaotic Behavior of $\prom$} \label{sec: chaotic}




In this section, we investigate the set of infinite standard Young tableaux of a fixed infinite shape $\mu$ and the jeu de taquin algorithm from a topological perspective by endowing $\SYT(\mu)$ with a metric.

\begin{theorem} \label{thm: Continuity of prom}
    Let $\mu$ be an infinite shape. 
    For $T,U\in\SYT(\mu)$, the following defines a metric on $\SYT(\mu)$:
    \[ \operatorname{d}(T,U) = \sum_{\{(i,j) \,\mid\, T_{i,j} \neq U_{i,j}\}} 2^{-(i+j)}. \]
    With respect to this metric,
    \[ \prom: \SYT(\mu) \rightarrow \SYT(\mu) \]
    is uniformly continuous.
\end{theorem}

\begin{proof}
    Verifying that $\operatorname{d}$ is a metric is straightforward. For example, the triangle inequality follows from the following observation: if $T,U,V\in\SYT(\mu)$ and $T$ and $V$ differ in the $(i,j)$ box, then either $T$ and $U$ or $U$ and $V$ also differ in the $(i,j)$ box. The rest of the argument is omitted. 
    
    For uniform continuity, 
    fix $\epsilon > 0$ and $N \in \N$ with $\sum_{(i,j)\in \mu \setminus (N^M)}  2^{-(i+j)} < \epsilon$,  where $M := N$ if $ \mu = (\infty^\infty)$ and $M := l$ if $\mu = (\infty^l)$. Choose $\delta > 0$ such that $\delta < 2^{-(N+M'+1)}$, where $M' := N+1$ if $\mu = (\infty^\infty)$ and $M' := l$ if $\mu = (\infty^l)$.
    Observe that if $ \operatorname{d}(T,U) < \delta$, then $T|_{((N+1)^{M'})} = U|_{((N+1)^{M'})}$.
    Hence $\prom(T)|_{(N^{M})} = \prom(U)|_{(N^{M})}$, and $ \operatorname{d}(\prom(T), \prom(U)) < \epsilon$.
\end{proof}

For an infinite shape $\mu$, excluding the trivial cases $(\infty^1)$ and $(1^\infty)$, we note that the topology on $\SYT(\mu)$ is not compact as it admits an embedding of $\N$ with the trivial subspace topology via any map that sends $n\in\N$ to some $T \in \SYT(\mu)$ with $T_{1,1} = 1$ and $T_{1,2} = n$.

\begin{definition}
    Let $(X,\operatorname{d})$ be a metric space.
    We say that a continuous map $f : X \rightarrow X$ is \emph{chaotic} in the sense of Devaney \cite{MR1046376} if the following conditions are satisfied:
    \begin{enumerate}
        \item $f$ is \emph{transitive}, i.e., for all non-empty open subsets $V,W$ of $X$, there is some $n \in \N$ such that $f^n(V) \cap W \neq \emptyset$,

        \item the periodic points of $f$ are dense in $X$, and

        \item $f$ has \textit{sensitive dependence on initial conditions}, i.e., there is some $\delta > 0$ such that for every $x \in X$ and every neighborhood $N$ of $x$, there is some $y \in N$ and some $n \in \N$ such that $\operatorname{d}(f^n(x),f^n(y)) > \delta$.
    \end{enumerate}
\end{definition}

While the sensitivity condition captures the essence of what it means to be chaotic, it is well known that it follows from the first two defining conditions.
We record this result as a theorem below.

\begin{theorem}[\cite{MR1157223}]
    If $f : X \rightarrow X$ is transitive and has dense periodic points, then $f$ has sensitive dependence on initial conditions.
\end{theorem}

Theorem \ref{thm: Continuity of prom} shows that, in particular, $\jdt : \SYT((\infty^l)) \rightarrow \SYT((\infty^l))$ is a continuous map between metric spaces.
It is natural to ask if jeu de taquin defines a chaotic system on infinite standard Young tableaux.
We answer this question in the affirmative.

\begin{theorem}\label{thm: jdt is chaotic}
    For $l \in \N$ and $\mu = (\infty ^l)$, $\prom : \SYT(\mu) \rightarrow \SYT(\mu)$ is chaotic in the sense of Devaney.
\end{theorem}

We prove Theorem \ref{thm: jdt is chaotic} by means of the following three lemmata.

\begin{lemma}\label{lem: restrictions are close}
    Let $l \in \N$, $\mu = (\infty^l)$, and $T,U \in \SYT(\mu)$.
    If $T|_{(m^l)} = U|_{(m^l)}$ for some $m \in \N$, then $\operatorname{d}(T,U) < 2^{-m}$.
\end{lemma}

    \begin{proof} We have
        \[
        d(T,U) = \sum_{\{(i,j) \,\mid\, T_{i,j} \neq U_{i,j}\}}2^{-(i+j)} \leq \sum_{(i,j) \in [1,l] \times[m + 1,\infty)}2^{-(i+j)}
        \]
        \[
        = \left(\sum_{i = 1}^l 2^{-i}\right)\left(\sum_{j=m+1}^\infty 2^{-j}\right) = \left(1 - 2^{-l}\right)2^{-m} < 2^{-m}.\qedhere
        \]
    \end{proof}

\begin{lemma}
    Let $l \in \N$ and $\mu = (\infty^l)$.
    There exists $T \in \SYT(\mu)$ such that for every $n \in \N$, $\lambda \vdash n$ with $\lambda'_1 \leq l$, and $S \in \SYT(\lambda)$ there is some $k \geq 0$ such that $S = \jdt^k(T)|_\lambda$.
    In particular, $\jdt : \SYT(\mu) \rightarrow \SYT(\mu)$ is transitive.
\end{lemma}

    \begin{proof}
        For each $m \in \N$, we may choose an enumeration $\{S(i,m) \mid 1 \leq i \leq N_m\}$ of $\SYT((m^l))$.
        Then
        \[T \coloneqq \biglsmash_{m=1}^\infty \left( \biglsmash_{i=1}^{N_m} S(i,m) \right) \in \SYT(\mu).\]
        If $n \in \N$, $\lambda \vdash n$ with $\lambda_1' \leq l$, and $S \in \SYT(\lambda)$, then there is some $1 \leq i \leq N_{\lambda_1}$ such that $S(i,\lambda_1)|_{\lambda} = S$.
        Thus, by Lemma \ref{lem: Non-rectangluar Smashing}, $\jdt^k(T)|_\lambda = S(i,\lambda_1)|_\lambda = S$ for $k = l\left(m(i-1)+\sum_{j=1}^{m-1}jN_j\right)$.

        Now, suppose $V_1,V_2$ are non-empty open subsets of $\SYT(\mu)$.
        For $i = 1,2$, choose $T_i \in V_i$, $\epsilon_i > 0$, and $m_i \in \N$ such that $\{U \in \SYT(\mu) \mid \operatorname{d}(T_i,U) < \epsilon_i\} \subseteq V_i$ and $2^{-m_i} < \epsilon_i$.
        Without loss of generality, choose $m_1 < m_2$.
        By the above, there exist $k_1,k_2 \geq 0$ such that $\jdt^{k_i}(T)|_{(m_i^l)} = T_i|_{(m_i^l)}$ for $i = 1,2$ and $k_1 < k_2$.
        Writing $U = \jdt^{k_1}(T)$, we see that $\operatorname{d}(T_1,U) < 2^{-m_1}< \epsilon_1$ and $\operatorname{d}(T_2,\jdt^{k_2 - k_1}(U)) < 2^{-m_2} < \epsilon_2$ by Lemma \ref{lem: restrictions are close}.
        Hence $U\in V_1$ and $\jdt^{k_2 - k_1}(U) \in \jdt^{k_2 - k_1}(V_1) \cap V_2$.
    \end{proof}



\begin{lemma}\label{lem: jdt has dense periodic pts}
    Let $l \in \N$ and $\mu = (\infty^l)$.
    Then $\jdt : \SYT(\mu) \rightarrow \SYT(\mu)$ has dense periodic points.
\end{lemma}

    \begin{proof}
        Let $T \in SYT(\mu)$ and $\epsilon > 0$.
        Choose $m \in \N$ large enough so that $2^{-m} < \epsilon$, and let $U = \biglsmash_{i=1}^\infty T|_{(m^l)}$.
        Observe that $U \in \SYT(\mu)$ is $ml$-periodic by applying Lemma \ref{lem: Non-rectangluar Smashing} $ml$ times, while $\operatorname{d}(T,U) < 2^{-m} < \epsilon$ by Lemma \ref{lem: restrictions are close}.
    \end{proof}

\section{Recurrence Behavior of $\prom$} \label{sec: recurrence}


As we have seen in Corollary \ref{Cor: doubly infinite is never periodic}, $\SYT((\infty^\infty))$ has no periodic tableaux. However, it does exhibit recurrent behavior which we now define.

\begin{definition} \label{Def: Recurrence}
    Let $\mu$ be a infinite shape and $T \in \SYT(\mu)$.
    We say that $T$ is \emph{recurrent} if for every $m,n \in \N$, there exists $k \in\N$ with $\prom^k(T)|_{(m^n)} = T|_{(m^n)}$.
\end{definition}

Note that it is straightforward to see from the definition of the metric in Theorem \ref{thm: Continuity of prom} (see also Lemma \ref{lem: restrictions are close}) that Definition \ref{Def: Recurrence} agrees with what is commonly called a \emph{recurrent point} or \emph{topological recurrence}. As the space $\SYT(\mu)$ is Hausdorff, in fact there are an infinite number of $k\in\N$ with $\prom^k(T)|_{(m^n)} = T|_{(m^n)}$. This form of recurrence is weaker than uniform recurrence (or almost periodic) in the sense that the set of return times may not be syndetic, i.e., the gap sizes between instances of recurrence may not be bounded.

The next definition is used to produce a recurrent element of $\SYT((\infty^\infty))$.

\begin{definition} \label{def: self replicating recurrent}
    Let $\mu=(\infty^\infty)$ and 
    \[ S = \vcenter{\hbox{\ytableaushort{12,34}}}\, .\]
    For $k\in\N$, recursively construct $S_k\in\SYT\big(((2^k)^{2^k})\big)$ as follows: begin with $S_1:=S$. If $S_k$ is constructed, define $S_{k+1}$ by
    \begin{enumerate}
    \item placing $S_k$ in the upper left quadrant of $((2^{k+1})^{2^{k+1}})$,
    \item placing $S_k+(2^{k})^2$ in the upper right quadrant of $((2^{k+1})^{2^{k+1}})$,
    \item placing $S_k+2(2^{k})^2$ in the lower left quadrant of $((2^{k+1})^{2^{k+1}})$, and
    \item placing $S_k+3(2^{k})^2$ in the lower right quadrant of $((2^{k+1})^{2^{k+1}})$.
    \end{enumerate}
    For example, $S_2$ is
    \[
        S_2 = 
        \vcenter{
          \hbox{
            \begin{ytableau}
              1 & 2 & *(lightgray) 5 & *(lightgray) 6 \\
              3 & 4 & *(lightgray) 7 & *(lightgray) 8 \\
              *(lightgray) 9 & *(lightgray) 10 & 13 & 14\\
              *(lightgray) 11 & *(lightgray) 12 & 15 & 16
            \end{ytableau}
          }
        }.
    \]
    Define $T\in\SYT(\mu)$ so that $T|_{((2^k)^{2^k})}=S_k$ for all $k\in\N$.
\end{definition}

The next result, Theorem \ref{thm: recurrent in infty^infty existence}, gives the existence of a recurrent element in $\SYT((\infty^\infty))$. It is straightforward to generalize this result to allow any starting rectangular tableau, $S$, by adjusting Definition \ref{def: self replicating recurrent} accordingly.

\begin{theorem} \label{thm: recurrent in infty^infty existence}
    Let $\mu=(\infty^\infty)$ and $T\in\SYT(\mu)$ constructed as in Definition \ref{def: self replicating recurrent}. Then $T$ is recurrent.
\end{theorem}

\begin{proof}
    Recall the definition of $S_k$, $k\in\N$, from Definition \ref{def: self replicating recurrent}.
    We claim that
    \[ \prom^{(2^{k})^2}(T)|_{((2^{k})^{2^{k}})} = T|_{((2^{k})^{2^{k}})}. \]
    To see this, simply observe that the incrementing in Definition \ref{def: self replicating recurrent} forces the incremented copy of $S_k$ in the upper right quadrant of $S_{k+1}$ to slide horizontally into the upper left quadrant under $\prom^{(2^{k})^2}$.

    Now let $m,n\in\N$. Finish the proof by choosing any $k\in\N$ so that $2^k \geq \max\{m,n\}$.
\end{proof}

Note that by taking a large enough rectangle in the proof of Theorem \ref{thm: recurrent in infty^infty existence} and observing that the result generalizes to any rectangular starting tableau, $S$, it is straightforward to see that there are recurrent extensions of any standard Young tableau.

The recurrence behavior in Theorem \ref{thm: recurrent in infty^infty existence} is quite beautiful, see Figure \ref{fig: recurrence figure 8 no. 1}. The pattern displayed shows that $\prom^k(T)|_{(4^4)}$ first cycles through $16$ distinct tableaux before returning to $T|_{(4^4)}$ when $k=16$. After that, it loops through $8$ additional distinct tableaux twice, returning to $T|_{(4^4)}$ for $k=24, 32$. This pattern then repeats for $k \mod 32$.
    \begin{figure}[H]
        \begin{center}
            \includegraphics[width=0.95\textwidth]{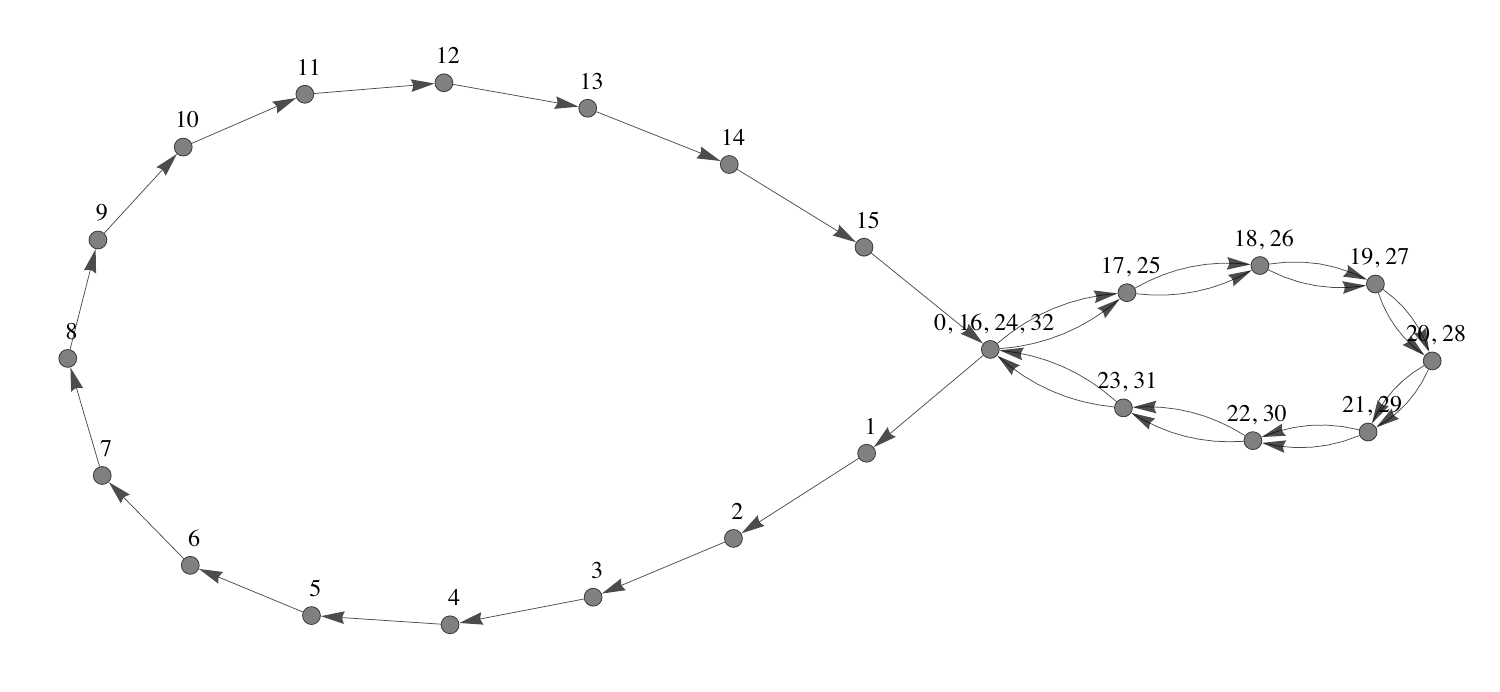}
        \end{center}
        \caption{$\prom^k(T)|_{(4^4)}$ for $k\in[0,32]$ from Theorem \ref{thm: recurrent in infty^infty existence}.}
        \label{fig: recurrence figure 8 no. 1}
    \end{figure}
More complicated patterns exist. For example, Figure \ref{fig: recurrence figure 8 no. 2} shows the recurrence pattern for the restriction of $\prom^k(T)$ to the shape $(2^3)$.
\begin{figure}[H]
        \begin{center}
            \includegraphics[width=0.85\textwidth]{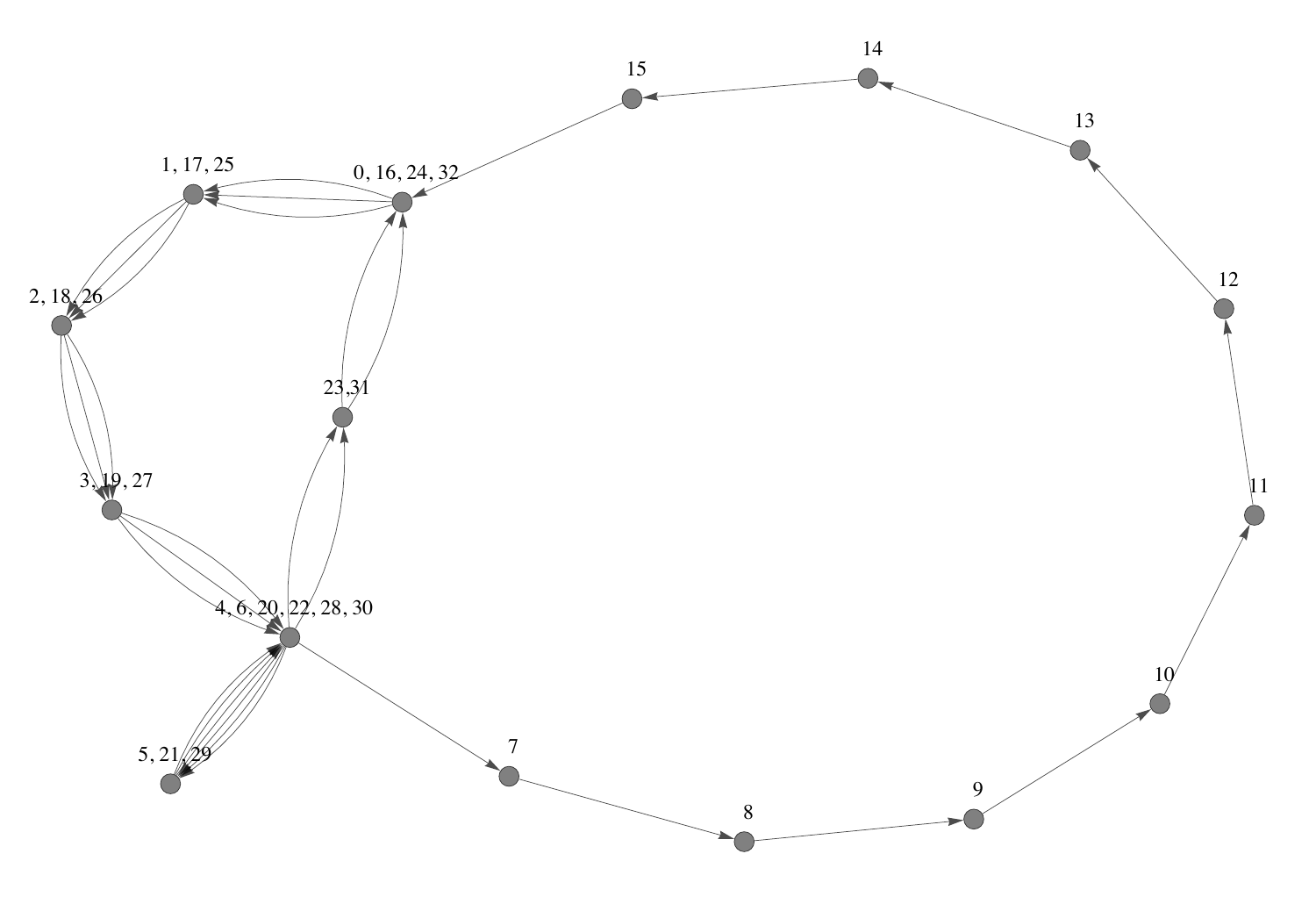}
        \end{center}
        \caption{$\prom^k(T)|_{(2^3)}$ for $k\in[0,32]$ from Theorem \ref{thm: recurrent in infty^infty existence}.}
        \label{fig: recurrence figure 8 no. 2}
    \end{figure}

\section{Open Questions and Concluding Remarks} \label{sec: concluding remarks}

Theorem \ref{thm: recurrent in infty^infty existence} provides an example of a recurrent $T$ for the shape $\mu=(\infty^\infty)$. It is not known if there exist $T$ with uniform recurrence, i.e., ones with syndetic return times.

It would be interesting to study the failure of Lemma \ref{lem: evenutal horizontal} for the shape $\mu = (\infty^\infty)$. In particular, does every tableau $T\in\SYT(\mu)$ have some $k\in\N$ such that the sliding path of $\prom^k(T)$ never becomes horizontal? 

Additionally, the failure of Theorem \ref{Thm: Size of Preimage} for the shape $\mu = (\infty^\infty)$ warrants further investigation, especially with respect to studying the possible range and frequencies of cardinalities of $\prom^{-1}(T)$. We note that there exist $T$ with $\prom^{-1}(T)$ uncountable. For example, it is possible to fill the infinite shape in Figure \ref{fig: jm} using the antidiagonals and paying attention to dominated entries (cf. Definition \ref{def: Dominated Entry}) to admit all the displayed possible sliding paths.
\begin{figure}[H]
        \begin{center}
        \includegraphics[width=0.5\textwidth]{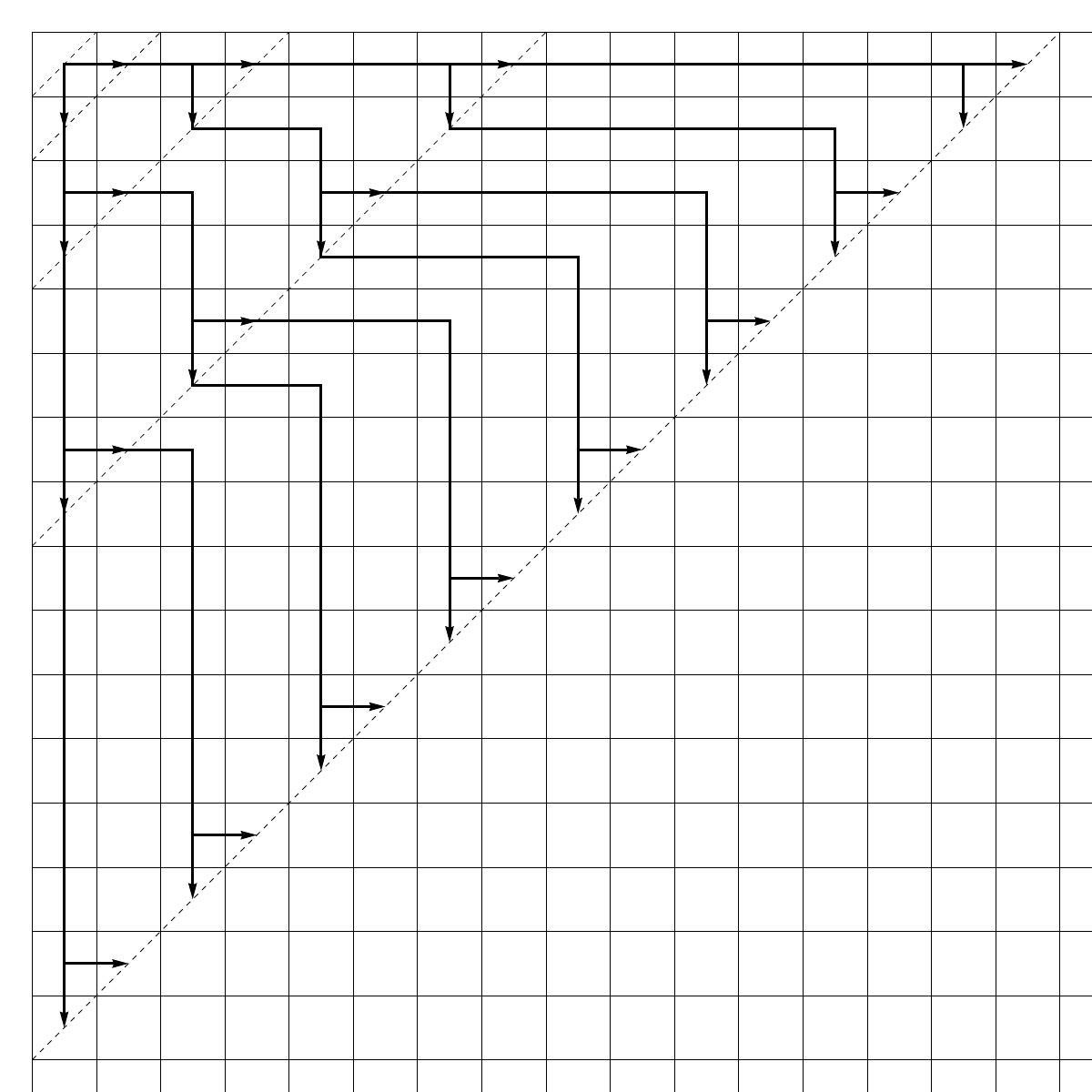}
        \end{center}
        \caption{Sliding paths for uncountable $\prom^{-1}(T)$.}
        \label{fig: jm}
\end{figure}

Theorem \ref{thm: Periodic gives linear growth} established that periodic tableaux with infinite shape exhibit linear growth in their entries. This naturally raises the question: do certain forms of recurrence (cf. Section~\ref{sec: recurrence}) impose growth conditions for the shape $\mu$?

More broadly, it would be valuable to develop a structure theorem for specific classes of recurrent tableaux, analogous to the classifications provided in Theorems \ref{thm: backfill and alt periodic classification} and \ref{Thm: Characterization of Pre-periodic}. Such a framework could deepen our understanding of the combinatorial and algebraic properties underlying these objects.

Finally, it would be fascinating to characterize the types of recurrence behavior illustrated in Figures \ref{fig: recurrence figure 8 no. 1} and \ref{fig: recurrence figure 8 no. 2}, and to construct tableaux that realize prescribed recurrence patterns. This leads to intriguing questions about which graphs can arise from such constructions, and how the tableau dynamics reflect the underlying graph structure.

\bibliographystyle{abbrvnat}
\bibliography{refs}

@article{Schutzenberger-Promotion-Original,
  title   = {Promotion des morphismes d'ensembles ordonn{\'e}s},
  journal = {Discrete Math.},
  volume  = {2},
  pages   = {73--94},
  year    = {1972},
  author  = {Sch{\"u}tzenberger, M. P.}
}

@inproceedings{Schutzenberger1977,
  author    = {Sch{\"u}tzenberger, M. P.},
  title     = {La correspondance de {R}obinson},
  booktitle = {Combinatoire et Repr{\'e}sentation du Groupe Sym{\'e}trique},
  series    = {Lecture Notes in Mathematics},
  volume    = {579},
  publisher = {Springer},
  pages     = {59--113},
  year      = {1977}
}

@book{fulton1997young,
  title     = {Young {T}ableaux: {W}ith {A}pplications to {R}epresentation {T}heory and {G}eometry},
  author    = {Fulton, William},
  series    = {London Mathematical Society Student Texts},
  volume    = {35},
  year      = {1997},
  publisher = {Cambridge University Press}
}

@article{RomikSniady2015,
  AUTHOR = {Romik, Dan and {\'S}niady, Piotr},
  TITLE = {Jeu de taquin dynamics on infinite {Y}oung tableaux and second class particles},
  JOURNAL = {Ann. Probab.},
  FJOURNAL = {The Annals of Probability},
  VOLUME = {43},
  YEAR = {2015},
  PAGES = {682--737},
  MRCLASS = {60C05 (05E10 37A05 60K35 82C22)},
  MRNUMBER = {3306003},
  MRREVIEWER = {Boyka L. Aneva}
}

@article{vervshik1977asymptotic,
  title   = {Asymptotic behavior of the {P}lancherel measure of the symmetric group and the limit form of {Y}oung tableaux},
  author  = {Ver{\v{s}}hik, Anatolii Moiseevich and Kerov, Sergei V.},
  journal = {Dokl. Akad. Nauk SSSR},
  volume  = {233},
  pages   = {1024--1027},
  year    = {1977}
}

@article{vershik1981asymptotic,
  title   = {Asymptotic theory of characters of the symmetric group},
  author  = {Vershik, Anatolii Moiseevich and Kerov, Sergei Vasil'evich},
  journal = {Funct. Anal. Appl.},
  volume  = {15},
  pages   = {246--255},
  year    = {1981}
}

@article{kerov1986characters,
  title   = {The characters of the infinite symmetric group and probability properties of the {R}obinson--{S}chensted--{K}nuth algorithm},
  author  = {Kerov, Sergei V. and Vershik, Anatolii M.},
  journal = {SIAM J. Algebraic Discrete Methods},
  volume  = {7},
  pages   = {116--124},
  year    = {1986}
}

@article{MR3200334,
  AUTHOR = {Clifford, Edward and Thomas, Hugh and Yong, Alexander},
  TITLE = {{$K$}-theoretic {S}chubert calculus for {${\rm OG}(n,2n+1)$} and jeu de taquin for shifted increasing tableaux},
  JOURNAL = {J. Reine Angew. Math.},
  FJOURNAL = {Journal f\"{u}r die Reine und Angewandte Mathematik. [Crelle's Journal]},
  VOLUME = {690},
  YEAR = {2014},
  PAGES = {51--63},
  MRCLASS = {05E10 (14N15)},
  MRNUMBER = {3200334},
  MRREVIEWER = {Li Li}
}

@article{MR2491941,
  AUTHOR = {Thomas, Hugh and Yong, Alexander},
  TITLE = {A jeu de taquin theory for increasing tableaux, with applications to {$K$}-theoretic {S}chubert calculus},
  JOURNAL = {Algebra Number Theory},
  FJOURNAL = {Algebra \& Number Theory},
  VOLUME = {3},
  YEAR = {2009},
  PAGES = {121--148},
  MRCLASS = {05E10 (14M15)},
  MRNUMBER = {2491941},
  MRREVIEWER = {Gregory S. Warrington}
}

@article{MR2806593,
  AUTHOR = {Thomas, Hugh and Yong, Alexander},
  TITLE = {The direct sum map on {G}rassmannians and jeu de taquin for increasing tableaux},
  JOURNAL = {Int. Math. Res. Not.},
  FJOURNAL = {International Mathematics Research Notices. IMRN},
  YEAR = {2011},
  VOLUME = {2011},
  PAGES = {2766--2793},
  MRCLASS = {14M15 (05E10 14C35)},
  MRNUMBER = {2806593},
  MRREVIEWER = {Li Li}
}

@article{MR1466956,
  AUTHOR = {Sheats, Jeffrey T.},
  TITLE = {A symplectic jeu de taquin bijection between the tableaux of {K}ing and of {D}e {C}oncini},
  JOURNAL = {Trans. Amer. Math. Soc.},
  FJOURNAL = {Transactions of the American Mathematical Society},
  VOLUME = {351},
  YEAR = {1999},
  PAGES = {3569--3607},
  MRCLASS = {05E15},
  MRNUMBER = {1466956},
  MRREVIEWER = {Ang{\`e}le M. Hamel}
}

@article{iwao2019jeu,
  title   = {Jeu de taquin, uniqueness of rectification and ultradiscrete {KP}},
  author  = {Iwao, Shinsuke},
  journal = {J. Integrable Syst.},
  volume  = {4},
  pages   = {xyz012, 25 pp.},
  year    = {2019},
}

@article{arnal2018sliding,
  title   = {Sliding Presentation of the Jeux de Taquin for Classical {L}ie Groups},
  author  = {Arnal, Didier and Khlifi, Olfa},
  journal = {Algebr. Represent. Theory},
  volume  = {21},
  pages   = {219--237},
  year    = {2018}
}

@article{van1998analogue,
  title   = {An analogue of Jeu de taquin for {L}ittelmann's crystal paths},
  author  = {van Leeuwen, Marc A. A.},
  journal = {S\'emin. Lothar. Comb.},
  volume  = {41},
  pages   = {B41b, 23 pp.},
  year    = {1998},
}

@article{rhoades2010cyclic,
  title   = {Cyclic sieving, promotion, and representation theory},
  author  = {Rhoades, Brendon},
  journal = {J. Comb. Theory, Ser. A},
  volume  = {117},
  pages   = {38--76},
  year    = {2010}
}

@article{petersen2009promotion,
  title   = {Promotion and cyclic sieving via webs},
  author  = {Petersen, T. Kyle and Pylyavskyy, Pavlo and Rhoades, Brendon},
  journal = {J. Algebr. Comb.},
  volume  = {30},
  pages   = {19--41},
  year    = {2009}
}

@article{purbhoo2013wronskians,
  title   = {Wronskians, cyclic group actions, and ribbon tableaux},
  author  = {Purbhoo, Kevin},
  journal = {Trans. Am. Math. Soc.},
  volume  = {365},
  pages   = {1977--2030},
  year    = {2013}
}

@article{levinson2017one,
  title   = {One-dimensional {S}chubert problems with respect to osculating flags},
  author  = {Levinson, Jake},
  journal = {Can. J. Math.},
  volume  = {69},
  pages   = {143--185},
  year    = {2017}
}

@article{gaetz2025web,
  title   = {Web bases in degree two from hourglass plabic graphs},
  author  = {Gaetz, Christian and Pechenik, Oliver and Pfannerer, Stephan and Striker, Jessica and Swanson, Joshua P.},
  journal = {Int. Math. Res. Not.},
  volume  = {2025},
  pages   = {rnaf189, 23 pp.},
  year    = {2025},
}

@article{alexandersson2021skew,
  title   = {Skew characters and cyclic sieving},
  author  = {Alexandersson, Per and Pfannerer, Stephan and Rubey, Martin and Uhlin, Joakim},
  journal = {Forum Math. Sigma},
  volume  = {9},
  pages   = {e41, 32 pp.},
  year    = {2021},
}

@article{pfannerer2020promotion,
  title   = {Promotion on oscillating and alternating tableaux and rotation of matchings and permutations},
  author  = {Pfannerer, Stephan and Rubey, Martin and Westbury, Bruce},
  journal = {Algebr. Comb.},
  volume  = {3},
  pages   = {107--141},
  year    = {2020}
}

@article{gaetz2023promotion,
  author  = {Gaetz, Christian and Pechenik, Oliver and Pfannerer, Stephan and Striker, Jessica and Swanson, John P.},
  title   = {Promotion permutations for tableaux},
  journal = {Comb. Theory},
  volume  = {4},
  pages   = {15, 56 pp.},
  year    = {2024},
}

@inproceedings{Roby2014,
  author    = {Roby, Thomas},
  title     = {Dynamical algebraic combinatorics and the homomesy phenomenon},
  booktitle = {Recent Trends in Combinatorics},
  series    = {The IMA Volumes in Mathematics and its Applications},
  volume    = {159},
  publisher = {Springer},
  year      = {2018},
  pages     = {619--652}
}

@article{Sniady2014,
  author  = {{\'S}niady, Piotr},
  title   = {Robinson--{S}chensted--{K}nuth algorithm, jeu de taquin, and {K}erov--{V}ershik measures on infinite tableaux},
  journal = {SIAM J. Discrete Math.},
  volume  = {28},
  pages   = {598--630},
  year    = {2014}
}

@article{MaslankaSniady2019,
  author  = {Ma{\'s}lanka, {\L}ukasz and {\'S}niady, Piotr},
  title   = {Second class particles and limit shapes of evacuation and sliding paths for random tableaux},
  journal = {Doc. Math.},
  volume  = {27},
  pages   = {2183--2273},
  year    = {2022}
}

@book{Sagan2001,
  AUTHOR    = {Sagan, Bruce E.},
  TITLE     = {The {S}ymmetric {G}roup},
  SERIES    = {Graduate Texts in Mathematics},
  VOLUME    = {203},
  EDITION   = {2nd},
  PUBLISHER = {Springer},
  YEAR      = {2001},
  PAGES     = {xvi+238},
  MRNUMBER  = {1824028}
}

@book{MR1046376,
  AUTHOR    = {Devaney, Robert L.},
  TITLE     = {An {I}ntroduction to {C}haotic {D}ynamical {S}ystems},
  SERIES    = {Studies in Nonlinearity},
  EDITION   = {2nd},
  PUBLISHER = {Addison-Wesley Publishing Company},
  YEAR      = {1989},
  PAGES     = {xviii+336},
  MRNUMBER  = {1046376}
}

@article{MR1157223,
  AUTHOR  = {Banks, J. and Brooks, J. and Cairns, G. and Davis, G. and Stacey, P.},
  TITLE   = {On {D}evaney's definition of chaos},
  JOURNAL = {Am. Math. Mon.},
  VOLUME  = {99},
  YEAR    = {1992},
  PAGES   = {332--334},
  MRNUMBER = {1157223}
}

@article{MR2515772,
  AUTHOR = {Stanley, Richard P.},
  TITLE = {Promotion and evacuation},
  JOURNAL = {Electron. J. Comb.},
  VOLUME = {16},
  YEAR = {2009},
  PAGES = {R9, 24 pp.},
  MRNUMBER = {2515772},
}
\end{document}